\documentclass[11pt,twoside,a4paper]{article}
\usepackage{amsmath}
\usepackage{amssymb}
\usepackage{amsthm}
\usepackage{algorithmic}
\usepackage{authblk}
\usepackage{cite}
\usepackage[english]{babel}
\usepackage[T1]{fontenc}
\usepackage{fullpage}
\usepackage{graphicx}
\usepackage{mathtools}
\usepackage{pgf}
\usepackage{tikz}
\usepackage{url}
\usepackage{caption}
\usepackage{hyperref}

\usepackage[noabbrev,capitalise,nameinlink]{cleveref}
\usepackage{fourier}
\usepackage{times}
\usepackage[activate={true,nocompatibility},
            final,tracking=false,
            kerning=true,
            spacing=true,
            factor=1100,
            stretch=10,
            shrink=10]{microtype}
\microtypecontext{spacing=nonfrench}
\DeclareMathAlphabet{\mathcal}{OMS}{cmsy}{m}{n}
\hypersetup{colorlinks,
            breaklinks,
            urlcolor=blue,
            linkcolor=blue,
            citecolor=blue}

\numberwithin{equation}{section}
\creflabelformat{equation}{\textup{(#1)}}

\newcommand{\stencilpt}[4][]{\node[circle,fill,draw,inner sep=2.5pt,label={above:#4},#1] at (#2) (#3) {}}
\newcommand{\bigstencilpt}[4][]{\node[circle,fill,draw,inner sep=4.5pt,label={[label distance=-3pt]90:#4},#1] at (#2) (#3) {}}
\newcommand{\stencilptund}[4][]{\node[circle,fill,draw,inner sep=2.5pt,label={[label distance=-3pt]270:#4},#1] at (#2) (#3) {}}
\newcommand{\bigstencilptund}[4][]{\node[circle,fill,draw,inner sep=4.5pt,label={[label distance=-3pt]270:#4},#1] at (#2) (#3) {}}
\newcommand{\stencilsq}[4][]{\node[rectangle,fill,draw,inner sep=3.5pt,label={above:#4},#1] at (#2) (#3) {}}

\definecolor{rednode}{rgb}{0.8,0.2,0.3}
\definecolor{greennode}{rgb}{0.2,0.8,0.3}
\definecolor{bluenode}{rgb}{0.3,0.2,0.8}

\newtheorem{assumption}{Assumption}
\newtheorem{definition}{Definition}
\newtheorem{algorithm}{Algorithm}
\newtheorem{theorem}{Theorem}
\newtheorem{lemma}{Lemma}

\title{Variational image regularization with Euler's elastica using a discrete gradient scheme}
\author{Torbj\o rn Ringholm\thanks{Department of Mathematical Sciences, Norwegian University of Science and Technology, N-7491 Trondheim, Norway.}
, Jasmina Lazi\'c\thanks{MathWorks, Matrix House, 10 Cowley Rd, Cambridge CB4 0HH, UK, and Mathematical Institute, Serbian Academy of Sciences and Arts, Kneza Mihaila 36, Beograd 11001, Serbia. }
, Carola-Bibiane Sch\"onlieb\thanks{Department of Applied Mathematics and Theoretical Physics (DAMTP), University of Cambridge, Wilberforce Road, Cambridge CB3 0WA, United Kingdom.}}

\begin{document}
\maketitle

\begin{abstract}
This paper concerns an optimization algorithm for unconstrained non-convex problems where the objective function has sparse connections between the unknowns. 
The algorithm is based on applying a dissipation preserving numerical integrator, the Itoh--Abe discrete gradient scheme, to the gradient flow of an objective function, guaranteeing energy decrease regardless of step size. 
We introduce the algorithm, prove a convergence rate estimate for non-convex problems with Lipschitz continuous gradients, and show an improved convergence rate if the objective function has sparse connections between unknowns. 
The algorithm is presented in serial and parallel versions.
Numerical tests show its use in Euler's elastica regularized imaging problems and its convergence rate and compare the execution time of the method to that of the iPiano algorithm and the gradient descent and Heavy-ball algorithms.

\end{abstract}

\section{Introduction}

A classic idea for minimizing a differentiable $V: \mathbb{R}^n \rightarrow \mathbb{R}$, $ n \geq 1$, is considering its gradient flow 
\begin{align}
\dot{\mathbf{u}}(t) = - \nabla V(\mathbf{u}(t)) \label{eq:gradflowintro}
\end{align}
and numerically integrating a solution along it. For example, the gradient descent algorithm can be easily derived from the explicit Euler scheme
\begin{align*}
\mathbf{u}^{k+1} - \mathbf{u}^k = -  \tau \nabla V(\mathbf{u}^k),
\end{align*} 
where $\tau$ is a step size and $\mathbf{u}^k$ an approximation to the value of $\mathbf{u}(k \tau)$. Other schemes can be used, such as Runge-Kutta and multistep methods. A discussion on step size conditions under which algebraically stable Runge-Kutta methods are dissipative can be found in \cite{hairer2013energy}. Even though these classes of ODE integrators are readily available, their use does not appear to have gained much traction in the optimization community. 

This disregard may be attributed to a division between the goals of numerical integration and numerical optimization; whereas the ODE integration schemes seek to approximate a solution path of \eqref{eq:gradflowintro} as accurately as possible, optimization schemes try to find a stationary point of \eqref{eq:gradflowintro} as quickly as possible. The former task requires small time steps while the latter task is generally completed more efficiently the larger the time steps are.
Thus, using regular ODE schemes to solve \eqref{eq:gradflowintro} is, in general, ineffective. 
However, in a recent article \cite{NIPS2017_6711}, the authors demonstrate that several well-known efficient optimization methods can be deduced from ODE integration schemes applied to equation \eqref{eq:gradflowintro}. Examples include Polyak's Heavy-ball method \cite{polyak1964some}, which may also be interpreted as a discretization of a gradient flow with an inertial term, and Nesterov's accelerated gradient method \cite{nesterov1983method}. 
These methods are equivalent to linear two-step methods with certain choices of step length. 
Also, the proximal point and proximal gradient methods \cite{nesterov2013gradient} are shown to correspond to an implicit Euler method and an Implicit-Explicit scheme, respectively. This gives credibility to the idea that ODE solvers with certain properties may indeed be useful as optimization schemes.

In recent years, new ODE solvers with properties well suited to optimization have emerged. In \cite{grimm_2016}, building on developments in the field of geometric integration, the authors apply discrete gradient schemes to the gradient flow of energy functionals arising from problems in variational image analysis. Discrete gradients, introduced in \cite{gonzalez96tia} and further studied in \cite{mclachlan99giu} have a property which is interesting from an optimization viewpoint; they are \textit{dissipativity preserving}. When applied to a dissipative ODE such as a gradient flow, the numerical solution is dissipative in the sense that
\begin{align*}
V(\mathbf{u}^{k+1}) \leq V(\mathbf{u}^k).
\end{align*}
The schemes thus convergence monotonously toward a critical point $V^*$ regardless of the step size used in the numerical integration if $V$ is continuously differentiable \cite{grimm_2016}.

While very efficient solvers exist for convex optimization problems, see e.g. \cite{Chambolle2011, nesterov1994interior, shor1985minimization}, the picture is different for non-convex optimization problems. Considerable effort has been spent in developing efficient schemes for classes of problems with special structure, e.g. problems with one convex but non-differentiable term and one non-convex but differentiable term \cite{Bolte2014,Ochs2014}. We will add to this effort by introducing a method based on the Itoh--Abe discrete gradient method \cite{Itoh1988} that is most effective when the objective function is continuously differentiable with sparsely connected unknowns. Indeed, in \cite{grimm_2016}, the authors use discrete gradient schemes with non-convex problems in mind, so this paper may be viewed as a continuation of their work.

A problem that fits the format of being non-convex with sparse connections is variational image analysis using a discretized Euler's elastica functional as a regularizer. Introduced in \cite{Nitzberg1993} for de-occluding objects in images, Euler's elastica regularization was further analyzed and applied to inpainting problems by Chan, Kang, and Shen in \cite{ChanKangShen}, who derive the Euler-Lagrange equations for the continuous Euler's elastica functional and solve these via finite difference schemes. This approach is not very computationally efficient, and attempts have been made to create more effective schemes, in particular in \cite{Tai2011} where an augmented Lagrangian approach was considered. This approach was later refined in \cite{zhu2013augmented}, \cite{Zhang2016}, \cite{bae2017augmented} and \cite{zhang_chen_deng_wang_2017}. Also of note are the approaches in \cite{BrediesPockWirth} and \cite{chambolle2017total} where convex approximations to the objective function are considered.

The method presented in \cref{alg:DG} resembles the coordinate gradient descent method, except that the gradient is approximated by a discrete gradient. 
It is derivative free and easy to use with only a step size parameter to choose. The scheme guarantees decrease at each iteration, at the expense of being implicit. 
A recent survey of coordinate descent algorithms and their convergence can be found in \cite{wright2015coordinate}. 
According to this survey, for coordinate descent methods one can expect $V(x^k) - V^* \in \mathcal{O}(1/k)$ for convex problems, with linear convergence if the problem is strongly convex. In \cite{CoordDescNesterov}, an accelerated coordinate descent method is presented, with $V(x^k) - V^* \in \mathcal{O}(1/k^2)$ for convex problems at the expense of computing a full vector operation for each coordinate update.
In the following, we will prove a convergence rate of $\min_{1 \leq j \leq k} \left\lbrace  \| \nabla V(\mathbf{u}^{j}) \|^2 \right\rbrace   \in \mathcal{O}(1/k^{1/2})$ for nonconvex, Lipschitz continuously differentiable problems. 
In \cite{erlend2018}, an $\mathcal{O}(1/k)$ convergence rate for convex, smooth problems and linear convergence for problems satisfying the Polyak-\L{}ojasiewicz condition \cite{karimi2016linear} are proved for several discrete gradient algorithms, including the one considered here. 
In the case of the Itoh--Abe discrete gradient, the rates have an $\mathcal{O}(n^{1/2})$ dependence on the problem size $n$; in \cref{lemma:MainLemmaSharper} we show that the rates can be improved when $V$ has sparsely connected unknowns. 
A convergence result for Itoh--Abe type discrete gradient schemes applied to non-differentiable, non-convex problems is presented in \cite{riis2018geometric}, with applications to parameter estimation problems. 

The paper is organized as follows: In the following section, we introduce discrete gradient methods for optimization and discuss the convergence rate and acceleration of a specific discrete gradient-type scheme. We also propose a method for adaptive time steps that allows acceleration of the algorithm for differentiable $V$, using four additional parameters. In section 3, the Euler's elastica regularization problem is introduced, and parallelization of the discrete gradient algorithm is discussed together with the effect of sparsity in the problem. Section 4 contains numerical experiments concerning the quality of denoising and inpainting, experimental convergence rates, the effect of coordinate ordering and problem size, execution time, and dependence on the initial condition. The final section summarizes the results.

\section{Discrete gradient methods}
The task at hand is to minimize a $\mathcal{C}^1$ functional $V: \mathbb{R}^n \rightarrow \mathbb{R}$, also called an energy, by solving its gradient flow
\begin{align}
\dot{\mathbf{u}} = - \nabla V(\mathbf{u}), \label{eq:gradflow}
\end{align}
where $\mathbf{u}(t) \in \mathbb{R}^n$ is the unknown and $\dot{\mathbf{u}}(t)$ denotes its time derivative. The reason for this is that $V$ dissipates along the flow of \eqref{eq:gradflow}; if $\mathbf{u}(t)$ solves \eqref{eq:gradflow}, then
\begin{align*}
\dfrac{\mathrm{d}}{\mathrm{d}t} V(\mathbf{u}(t)) = \left\langle \dot{\mathbf{u}}, \nabla V (\mathbf{u}(t)) \right\rangle= - \| \nabla V (\mathbf{u}(t)) \|^2 \leq 0,
\end{align*}
where $\|\cdot \|$ denotes the Euclidian norm on $\mathbb{R}^n$. Due to the dissipation, $\mathbf{u}(t)$ approaches a critical point of $V$ as $t \rightarrow \infty$ as long as $V$ is bounded from below. In general, $V$ is nonlinear such that numerical schemes must be employed to solve \eqref{eq:gradflow} until a large stopping time $T$. This gives rise to different optimization algorithms depending on the scheme used. For example, a forward Euler scheme results in the gradient descent method. The forward Euler method has several drawbacks, one being that choosing too large step sizes results in instability. This necessitates step size selection, which may result in impractically small steps considering that we wish to obtain a stationary point of \eqref{eq:gradflow}. It is therefore of interest to investigate the use of numerical schemes that have lenient step size restrictions or none at all. One such class of schemes is called \textit{discrete gradient} methods. Discrete gradients were introduced in \cite{gonzalez96tia} to unite several energy preserving and dissipative ODE solvers under a single label. A seminal paper \cite{mclachlan99giu} covers their use as ODE solvers and which ODEs they are applicable to. 

\begin{definition} \label{def:DG}
Given a differentiable function $V: \mathbb{R}^n \rightarrow \mathbb{R}$, we say that $\overline{\nabla} V  : \mathbb{R}^n \times \mathbb{R}^n \rightarrow \mathbb{R}^n$ is a discrete gradient of $V$ if it is continuous and for all $\mathbf{u}, \mathbf{v} \in \mathbb{R}^n$,
\begin{align*}
\left\langle \overline{\nabla} V(\mathbf{u},\mathbf{v}),\mathbf{v} - \mathbf{u} \right\rangle&= V(\mathbf{v}) - V(\mathbf{u}),\\
\lim_{\mathbf{v} \rightarrow \mathbf{u}}\overline{\nabla} V(\mathbf{u},\mathbf{v}) &= \nabla V(\mathbf{u}).
\end{align*} 
\end{definition}
Discrete gradients can be used in schemes to solve \eqref{eq:gradflow} numerically by computing
\begin{align}
\mathbf{u}^{k+1} = \mathbf{u}^k - \tau_k \overline{\nabla} V(\mathbf{u}^k,\mathbf{u}^{k+1}), \label{eq:DGscheme}
\end{align}
where $\tau_k > 0$ is the step size at iteration number $k$. A key property is that the scheme is dissipating; by \cref{def:DG} and the scheme \eqref{eq:DGscheme}, we have
\begin{align}
V(\mathbf{u}^{k+1}) - V(\mathbf{u}^k) =  \left\langle\overline{\nabla} V(\mathbf{u}^k,\mathbf{u}^{k+1}),\mathbf{u}^{k+1} - \mathbf{u}^k \right\rangle
 = -\dfrac{1}{\tau_k} \| \mathbf{u}^{k+1} - \mathbf{u}^k \|^2.  \label{eq:DGdissip}
\end{align}
Note that the dissipation property holds regardless of the step size $\tau_k$. 

\Cref{def:DG} is quite broad and as a result, there exist several types of discrete gradients. Two popular choices are the midpoint discrete gradient \cite{gonzalez96tia}
and the Average Vector Field (AVF) discrete gradient \cite{Harten1983}.
They give second-order accurate schemes for \eqref{eq:gradflow} and are suited for solving ODEs precisely. In our case, solving \eqref{eq:gradflow} as exactly as possible is not the main concern, rather, we need a scheme with cheap time steps and fast convergence toward a minimizer. 
The schemes obtained using the Gonzalez and AVF discrete gradients are fully implicit in the sense that in general, at each time step of \eqref{eq:DGscheme}, a single $n$-dimensional system of nonlinear equations must be solved. For large $n$, this is slow since the complexity of solving such a system is typically $\mathcal{O}(n^2)$. 
Instead, we consider the Itoh--Abe discrete gradient \cite{Itoh1988}, defined componentwise as 
\begin{align*}
(\overline{\nabla} V(\mathbf{u},\mathbf{v}))_l= \dfrac{V\left(\mathbf{u} + \sum \limits_{j=1}^l (v_{j} - u_{j} )\mathbf{e}_{j} \right) - V\left(\mathbf{u} + \sum \limits_{j=1}^{l-1} (v_{j} - u_{j})\mathbf{e}_{j}\right)}{v_{l} - u_{l} },
\end{align*}
where $\mathbf{e_j}$ denotes the $j$'th standard basis vector. This discrete gradient, while still implicit, has two advantages over the Gonzales and AVF discrete gradients. First, its use in the scheme \eqref{eq:DGscheme} requires the solution of $n$ scalar nonlinear equations per time step, meaning its computational complexity scales as $\mathcal{O}(n)$. 
Secondly, it is derivative-free and requires only computations of differences between the objective function with variation in one variable, which may be considerably cheaper than evaluating the objective function itself; consider for example the optimization problem
\begin{align*}
\min \limits_{\mathbf{u} \in \mathbb{R}^n} \left\lbrace V(\mathbf{u}) = \sum_{i= 1}^M V_i(\mathbf{u}) \right\rbrace,
\end{align*}
where at most $N$ of the $V_i$ depend on a given coordinate, say, $u_k$. Then, computing the difference $V(\mathbf{u} + \mathbf{e}_k u_k) - V(\mathbf{u}) $ amounts to computing at most $N$ values of the $V_i$, a cost comparable to that of calculating one coordinate derivative of $V$.

\subsection{The algorithm}

The algorithm based on using the Itoh--Abe discrete gradient with fixed step size $\tau_k = \tau$ in \eqref{eq:DGscheme} is presented in \cref{alg:DG}. As a stopping criterion we set a tolerance $tol$ and stop when $(V(\mathbf{u}^k) - V(\mathbf{u}^{k-1}))/V(\mathbf{u}^0) < tol$. This criterion is economical to evaluate, requiring no evaluation of $V$ since the energy increments are known from the dissipation property \eqref{eq:DGdissip}. 
\begin{algorithm}[DG]
\begin{algorithmic}
\label{alg:DG}
\STATE{}
\STATE{$\text{Choose } \tau > 0, \, tol > 0 \text{ and } \mathbf{u}^0 \in \mathbb{R}^n. \text{ Set } k = 0. $}
\REPEAT 
\STATE{$\mathbf{v}_0^k = \mathbf{u}^k $}
\FOR{$j = 1,...,n$}
\STATE{Solve $\beta_j^k  = -\tau (V(\mathbf{v}_{j-1}^{k} + \beta_j^k \mathbf{e}_j) - V(\mathbf{v}_{j-1}^{k}))/\beta_j^k$}
\STATE{$\mathbf{v}_j^k = \mathbf{v}_{j-1}^{k} + \beta_j^k \mathbf{e}_j$}
\ENDFOR
\STATE{$\mathbf{u}^{k+1} = \mathbf{v}_{n}^k$}
\STATE{$k = k + 1$}
\UNTIL{$\left(V(\mathbf{u}^k) - V(\mathbf{u}^{k-1})\right)/V(\mathbf{u}^0) < tol$}
\end{algorithmic}
\end{algorithm}
To solve the nonlinear scalar subproblems defining the $\beta_j^k$, we use the Brent-Dekker algorithm \cite{brent1971algorithm}. This is a derivative free method based on a combination of bisection and interpolation algorithms which converges superlinearly if the function whose root is to be found is $\mathcal{C}^1$ near the root. It is the method of choice for scalar root finding problems in \cite{press2007numerical}.

The algorithm can be accelerated through adaptive step sizes. Unlike line search methods, each time step is implicit and so changing $\tau_k$ requires a recomputation of $\mathbf{u}^{k+1} $, which can be costly and should be avoided. One way of adapting $\tau_k$, which can be used for differentiable $V$, is to check conditions similar to the Wolfe conditions \cite{wolfe1969convergence}. We consider, with constants $c_1 \in (0,1)$ and $c_2 \in (c_1,1)$ the conditions
\begin{align}
V(\mathbf{u}^{k+1}) - V(\mathbf{u}^{k}) &\leq c_1 \left\langle \nabla V(\mathbf{u}^{k}),\mathbf{u}^{k+1} - \mathbf{u}^{k}\right\rangle, \label{eq:wolfe1}\\
\left\langle \nabla V(\mathbf{u}^{k+1}), \mathbf{u}^{k+1} - \mathbf{u}^{k}\right\rangle &\geq c_2 \left\langle\nabla V(\mathbf{u}^{k}), \mathbf{u}^{k+1} - \mathbf{u}^{k}\right\rangle . \label{eq:wolfe2}
\end{align}
If condition \eqref{eq:wolfe1} holds, regardless of whether \eqref{eq:wolfe2} holds, then $\tau_k$ is increased for the next iteration by a factor $\lambda$ > 1. If \eqref{eq:wolfe1} does not hold but \eqref{eq:wolfe2} does, one takes $\tau_{k+1} = \rho\tau_{k}$ where $\rho \in (0,1)$. If neither condition holds, the step size is not changed. In all cases, the new value $\mathbf{u}^{k+1}$ is accepted. With this approach one obtains step sizes that are adjusted based on prior performance while not wasting previous computations, summed up in \cref{alg:DG-ADAPT}. 
\begin{algorithm}[DG-ADAPT]
\begin{algorithmic}
\label{alg:DG-ADAPT}
\STATE{}
\STATE{$\text{Choose } \tau_0 > 0, \, tol > 0, \, \rho \in (0,1), \, \lambda > 1, \, c_1 \in (0,1), \, c_2 \in (c_1,1), \text{ and } \mathbf{u}^0 \in \mathbb{R}^n.$}
\STATE{$\text{Set } k = 0. $}
\REPEAT 
\STATE{$\mathbf{v}_0^k = \mathbf{u}^k $}
\FOR{$j = 1,...,n$}
\STATE{Solve $\beta_j^k  = -\tau_k (V(\mathbf{v}_{j-1}^{k} + \beta_j^k \mathbf{e}_j) - V(\mathbf{v}_{j-1}^{k}))/\beta_j^k$}
\STATE{$\mathbf{v}_j^k = \mathbf{v}_{j-1}^{k} + \beta_j^k \mathbf{e}_j$}
\ENDFOR
\STATE{$\mathbf{u}^{k+1} = \mathbf{v}_{n}^k$}
\IF{$V(\mathbf{u}^{k+1}) - V(\mathbf{u}^{k}) \leq  c_1  \left\langle \nabla V(\mathbf{u}^{k}), \mathbf{u}^{k+1} - \mathbf{u}^{k}\right\rangle$}
\STATE{$\tau_{k+1} = \lambda\tau_k$}
\ELSIF {$\left\langle\nabla V(\mathbf{u}^{k+1}), \mathbf{u}^{k+1} - \mathbf{u}^{k} \right\rangle \geq c_2  \left\langle \nabla V(\mathbf{u}^{k}), \mathbf{u}^{k+1} - \mathbf{u}^{k} \right\rangle$}
\STATE{$\tau_{k+1} = \rho\tau_k$}
\ENDIF
\STATE{$k = k + 1$}
\UNTIL{$\left(V(\mathbf{u}^k) - V(\mathbf{u}^{k-1})\right)/V(\mathbf{u}^0) < tol$}
\end{algorithmic}
\end{algorithm}
Condition \eqref{eq:wolfe2} provides a lower bound on the step size when $\nabla V$ is Lipschitz continuous with Lipschitz constant $L$.
Firstly, since $\nabla V$ is Lipschitz, the descent lemma \cite[Proposition A.24]{bertsekas1999nonlinear}, provides the estimate
\begin{align}
V(\mathbf{u}) &\leq V(\mathbf{v}) + \left\langle \nabla V(\mathbf{v}), \mathbf{u} - \mathbf{v} \right\rangle + \dfrac{L}{2} \|\mathbf{v} - \mathbf{u} \|^2 \label{eq:descentLemmaPure}
\end{align}
which holds for all $\mathbf{u},\mathbf{v} \in \mathbb{R}^n$. Combining \eqref{eq:descentLemmaPure} with \eqref{eq:wolfe2} and \eqref{eq:DGdissip} we find
\begin{align*}
c_2\left\langle\nabla V(\mathbf{u}^{k}), \mathbf{u}^{k} - \mathbf{u}^{k+1}\right\rangle &\geq \left\langle \nabla V(\mathbf{u}^{k+1}), \mathbf{u}^{k} - \mathbf{u}^{k+1}\right\rangle \\
& \geq V(\mathbf{u}^{k}) - V(\mathbf{u}^{k+1}) - \frac{L}{2}\|\mathbf{u}^{k+1} - \mathbf{u}^{k} \|^2\\
& = \left(1 - \frac{L}{2}\tau_k \right)\left(V(\mathbf{u}^{k}) - V(\mathbf{u}^{k+1})\right).
\end{align*}
Rearranging and applying \eqref{eq:descentLemmaPure} once more, we find
\begin{align*}
\frac{L}{2}\tau_k \left(V(\mathbf{u}^{k}) - V(\mathbf{u}^{k+1})\right) \geq (1-c_2)\left(V(\mathbf{u}^{k}) - V(\mathbf{u}^{k+1})\right) -  \frac{c_2L}{2}\|\mathbf{u}^{k+1} - \mathbf{u}^{k} \|^2.
\end{align*}
Using \eqref{eq:DGdissip} on the last term to eliminate $V(\mathbf{u}^{k}) - V(\mathbf{u}^{k+1})$ we find the lower bound
\begin{align*}
\tau_k \geq \frac{1-c_2}{1+c_2} \frac{2}{L}.
\end{align*}

\subsection{Convergence}
\label{sect:theory}

In \cite{grimm_2016}, the authors prove that the iterates of \cref{alg:DG} converge toward a critical point, but do not estimate the convergence rate. 
\cref{theo:Convergence} concerns the convergence rate of \cref{alg:DG-ADAPT} in the general case of a non-convex objective $V$.
A sublinear convergence rate of $\mathcal{O}(1/k)$ for convex problems and a linear convergence rate for problems satisfying the Polyak-\L{}ojasiewicz inequality are given in \cite{erlend2018}. 
These theorems are stated below without proof to support the discussion of numerical results presented in \cref{sect:Numerical}, where better rates than the ones proved in \cref{theo:Convergence} are observed when choosing $\epsilon$ in the smoothing of the Euler's elastica regularizer large enough.
The following assumption is common to these theorems and \cref{lemma:MainLemmaSharper} in the subsequent section. Here, by coordinate Lipschitz continuity of $f:\mathbb{R}^n \rightarrow \mathbb{R}^n$, we mean that for $j = 1,...,n$, $| f_j(\mathbf{x}) - f_j(\mathbf{y})| \leq L_j \|\mathbf{x}-\mathbf{y}\|$.

\begin{assumption}\label{ass:MainAssumption}
The function $V: \mathbb{R}^n \rightarrow \mathbb{R}$ is $\mathcal{C}^1$, bounded from below and coercive. Furthermore, $\nabla V$ is Lipschitz with Lipschitz constant $L$ and coordinatewise Lipschitz constants $L_j \in [L_{\min}, L_{\max}]$, and all time steps $\tau_{j}$  lie in $[\tau_{\min},\tau_{\max}] \subset \mathbb{R}^+$.
\end{assumption}

\noindent The following proof is inspired by that of \cite[Lemma 3.3]{beck2013convergence}, and bears resemblance to the classical theorem of Zoutendijk for line search methods \cite[Theorem 3.2]{nocedal2006numerical}, with the exception that it does not rely on a Wolfe condition for step sizes and explicitly states a convergence rate.

\begin{theorem}\label{theo:Convergence} 
If \cref{ass:MainAssumption} holds, the $\mathbf{u}^k$ produced by \cref{alg:DG-ADAPT} satisfy
\begin{align*}
\min \limits_{1 \leq j \leq k} \left\lbrace  \| \nabla V(\mathbf{u}^{j}) \|^2 \right\rbrace   \leq \nu  \dfrac{V(\mathbf{u}^0) - V^*}{k}, \quad \nu = 2  L_{\max}^2 \left( \tau_{\max}n + \dfrac{\tau_{\max}}{L_{\max}^2\tau_{\min}^2}   \right)  
\end{align*}
where $V^* > -\infty$ is a local minimum.
\end{theorem}
\begin{proof}
The coordinatewise Itoh--Abe scheme, with $\beta_l^k = u^{k+1}_l - u^k_l$, reads
\begin{align*}
\beta_l^k = - \tau_k \dfrac{V\left(\mathbf{v}^{k}_{l} \right) - V\left(\mathbf{v}^{k}_{l-1}\right)}{\beta_l^k },
\end{align*}
with $\mathbf{v}^{k}_{l} := \mathbf{u}^k + \sum_{j=1}^l \beta_j^k\mathbf{e}_{j}$ such that $\mathbf{u}^{k+1} = \mathbf{v}_n^{k}$. By the triangle inequality,
\begin{align*}
\left| \dfrac{\partial V}{\partial u_{l}}(\mathbf{u}^{k+1}) \right| &\leq \left| \dfrac{\partial V}{\partial u_{l}}(\mathbf{u}^{k+1}) - \dfrac{V(\mathbf{v}^{k}_{l} ) - V(\mathbf{v}^{k}_{l-1})}{\beta_{l}^{k} }  \right| + \dfrac{1}{\tau_k} \left| \beta_{l}^{k} \right|.
\end{align*}
Since $V \in \mathcal{C}^1$, the mean value theorem holds, meaning
\begin{align*}
\dfrac{V(\mathbf{v}^{k}_{l} ) - V(\mathbf{v}^{k}_{l-1})}{\beta_{l}^{k} } = \dfrac{\partial V}{\partial u_{l}}(\mathbf{v}^{k}_{l-1} + s\beta_{l}^{k} \mathbf{e}_{l})
\end{align*}
for some $s \in (0,1)$. Hence,
\begin{align}
\left| \dfrac{\partial V}{\partial u_{l}}(\mathbf{u}^{k+1}) \right| &\leq \left| \dfrac{\partial V}{\partial u_{l}}(\mathbf{u}^{k+1}) -  \dfrac{\partial V}{\partial u_{l}}(\mathbf{v}^{k}_{l-1} + s\beta_{l}^{k} \mathbf{e}_{l})  \right|  +  \dfrac{1}{\tau_k}\left| \beta_{l}^{k} \right|.
\label{eq:startingpoint}
\end{align}
Exploiting the coordinatewise Lipschitz continuity of the gradient, we have
\begin{align*}
\left| \dfrac{\partial V}{\partial u_{l}}(\mathbf{u}^{k+1}) \right| &\leq L_l \| \mathbf{u}^{k+1} - \mathbf{v}^{k}_{l-1} - s\beta_{l}^{k} \mathbf{e}_{l} \| + \dfrac{1}{\tau_k}\left| \beta_{l}^{k} \right|. 
\end{align*}
Squaring this and summing over all coordinates, we get
\begin{align*}
\| \nabla V(\mathbf{u}^{k+1} )\|^2  &\leq \sum \limits_{l = 1}^{n} \left(L_l \| \mathbf{u}^{k+1} - \mathbf{v}^{k}_{l-1} - s\beta_{l}^{k}\mathbf{e}_{l} \| + \dfrac{1}{\tau_k}\left| \beta_{l}^{k} \right| \right)^2 \\
& \leq 2  L_{\max}^2 \left( n + \dfrac{1}{L_{\max}^2 \tau_{\min}^2}  \right) \| \mathbf{u}^{k+1} - \mathbf{u}^{k} \|^2\\
& \leq \nu \left(V(\mathbf{u}^{k}) - V(\mathbf{u}^{k+1})\right).
\end{align*}
We then find 
\begin{align*}
k \min \limits_{1 \leq j \leq k} \left\lbrace \| \nabla V(\mathbf{u}^{j}) \|^2 \right\rbrace  \leq \sum \limits_{j=1}^{k}\| \nabla V(\mathbf{u}^{j} )\|^2 \leq \nu  \left(V(\mathbf{u}^{0}) - V(\mathbf{u}^{k})\right) \leq \nu  \left(V(\mathbf{u}^{0}) - V^{*}\right),
\end{align*}
which concludes the proof.
\end{proof}
\textbf{\textit{Remark:}} We can choose a fixed $\tau_k = \tau$ that minimizes $\nu$, yielding
\begin{align*}
\tau = \dfrac{1}{L_{\max}\sqrt{n}}, \quad \nu = 4L_{\max}\sqrt{n}.
\end{align*}
Thus, the complexity of the above bound with respect to the problem size $n$ is $\mathcal{O}(n^{1/2})$. Furthermore, we can obtain bounds of the type $\| \nabla V(\mathbf{u}^{k+1} )\|^2 \leq \nu \left(V(\mathbf{u}^{k}) - V(\mathbf{u}^{k+1})\right)$ for other discrete gradients, yielding similar convergence rates. Such bounds are shown in \cite[Lemma 5.1]{erlend2018}.

As with descent methods, the convergence rate improves with additional assumptions on $V$, in particular assuming that $V$ is convex. 
We state the following theorems, proved in \cite{erlend2018} and inspired by those in \cite{beck2013convergence}, for later reference. Similarly to the proof of \cref{theo:Convergence}, they are are based on bounding $\| \nabla V(\mathbf{u}^{k+1} )\|^2 \leq \nu \left(V(\mathbf{u}^{k}) - V(\mathbf{u}^{k+1})\right)$ and so the factor $\nu$ appears here as well.

\begin{theorem}\label{theo:ConvergenceConv}
If \cref{ass:MainAssumption} holds and $V$ is in addition convex, the iterates $\mathbf{u}^k$ produced by \cref{alg:DG-ADAPT} satisfy, with $\nu$ as in \cref{theo:Convergence},
\begin{align*}
V(\mathbf{u}^{k}) - V^* \leq \dfrac{\nu R(\mathbf{u}^0)^2 }{k + 2\nu/L }.
\end{align*}
where $V^*$ is a minimum and $R(\mathbf{u}^0)$ is the diameter of $\{\mathbf{u} \in \mathbb{R}^n | V(\mathbf{u}) \leq V(\mathbf{u^0})\}$.
\end{theorem}
The next theorem concerns the convergence rate of \cref{alg:DG-ADAPT} when $V$ is a P\L{}-function,  i.e. $V$ satisfies the Polyak-\L{}ojasiewicz inequality with parameter $\sigma$,  
\begin{align*}
\frac{1}{2}\| \nabla V(\mathbf{u}) \|^2 \leq \sigma(V(\mathbf{u}) - V^*).
\end{align*}
Note that under \cref{ass:MainAssumption}, all strongly convex functions are P\L{}-functions \cite{karimi2016linear}.
\begin{theorem}\label{theo:ConvergenceStrConv}
If \cref{ass:MainAssumption} holds and $V $ is a P\L{}-function, the iterates of \cref{alg:DG-ADAPT} satisfy, with $\nu$ as in \cref{theo:Convergence},
\begin{align*}
V(\mathbf{u}^{k}) - V^* \leq \left(1 - \dfrac{2\sigma}{\nu} \right)^k(V(\mathbf{u}^{0}) - V^*).
\end{align*}
\end{theorem}
\textit{\textbf{Remark:}} The above theorems mean that for convex problems, too, the algorithm has a worst-case complexity of $\mathcal{O}(n^{1/2})$ with respect to the problem dimension $n$, compared to $\mathcal{O}(n^{3/2})$ for the cyclic coordinate descent algorithm \cite{wright2015coordinate} and $\mathcal{O}(n)$ for the expected bounds of stochastic coordinate descent \cite{CoordDescNesterov}.
We shall see in $\cref{lemma:MainLemmaSharper}$ that the complexity can be reduced depending on a sparsity property of $V$.

\section{The Euler's elastica problem}
We will use \cref{alg:DG} for variational image analysis with Euler's elastica regularization. 
In variational image analysis one repairs a damaged input greyscale image $g: \Omega \rightarrow [0,1]$, where $\Omega \subset \mathbb{R}^2$ is often rectangular, by finding an output image $u: \Omega \rightarrow [0,1]$ that minimizes a functional
\begin{align}
V_c(u) = d_c(K_c u,g) + \alpha J_c(u). \label{eq:contenergy}
\end{align}
Here, $K_c$ is a forward operator relating $u$ to $g$, $d_c$ a function measuring the distance between $K_cu$ and $g$, $J_c$ a regularization functional and $\alpha > 0$ a constant. The subscript $c$ emphasizes that the functions are continuous; they will later be discretized and renamed. When $J_c$ is the Euler's elastica energy below, $\alpha$ is included in $a$ and $b$. 

The forward operator $K_c$, which may be linear, is inherent to the problem. 
For example, when considering an inpainting problem where the goal is to interpolate $g$ in a subset $D$ of the image domain $\Omega$ in which there is no given data, one would take $K_c$ as a restriction to $\Omega \backslash D$. Since this leaves $K_c u$ undefined in $\Omega$, the fidelity term should only compare with values of $g$ on $\Omega \backslash D$. 
This has the effect of maintaining fidelity only in areas where the image is known, at the cost of generating an ill-posed problem due to non-unique solutions. In the denoising problem, where random noise is added to an image in unknown pixels, the usual choice is to take $K_c$ as the identity operator since there is no information about which pixels are damaged.

The terms that differentiate approaches to image analysis are the $d_c$ and $J_c$ functions, and the implementation of the $K_c$ operator if applicable. One often takes $d_c$ as an $L^p$ metric, while $J_c$ can be chosen in several ways. A popular choice is the total variation (TV) \cite{Rudin1992} regularization which, for differentiable $u$, can be stated as
\begin{align*}
J_{TV}(u) = \int \limits_{\Omega} |\nabla u| \mathrm{d}\mathbf{x}.
\end{align*}
In practice, one often wishes to work with a differentiable function after discretizing $J$, thus using a smoothed version of $J_{TV}$, with $0 < \epsilon \ll 1$, given as
\begin{align*}
J_{TV_\epsilon}(u) =  \int \limits_{\Omega} |\nabla u|_{\epsilon} \mathrm{d}\mathbf{x} = \int \limits_{\Omega} \sqrt{ \dfrac{\partial u}{\partial x}^2 + \dfrac{\partial u}{\partial y}^2 + \epsilon} \, \mathrm{d}\mathbf{x}.
\end{align*}
This is the $\mathrm{TV}_{\epsilon}$ regularizer. The Euler's elastica regularizer generalizes $J_{TV}$, adding a curvature dependent term. It is stated for $\mathcal{C}^2(\Omega)$ functions $u$ as \cite{ChanKangShen}
\begin{align*}
J_c(u) = \int \limits_{\Omega} \left( a + b\left( \nabla \cdot \dfrac{\nabla u}{|\nabla u|} \right)^2 \right)|\nabla u| \mathrm{d}\mathbf{x}, 
\end{align*}
where $a,b > 0$. We will consider the smoothed version
\begin{align}
J_{\epsilon}(u) = \int \limits_{\Omega} C(u)G(u) \mathrm{d}\mathbf{x}, \quad C(u) =  a + b\left( \nabla \cdot \dfrac{\nabla u}{|\nabla u|_{\epsilon}} \right)^2, \quad G(u) = |\nabla u|_{\epsilon}. \label{eq:EulerElasticaSmooth}
\end{align}
where $C(u)$ and $G(u)$ are smoothed curvature and gradient terms.

\subsection{Discretization}
In the following, for an image with resolution $n_x \times n_y$ we take $\Omega = [0,1]\times[0,n_y/n_x]$, such that each pixel represents the value over an $h\times h$ area, where $h = 1/n_x$.
Thus, with a discrete input image $\mathbf{g}$ indexed as $g_{ij}$ at points $\mathbf{x}_{ij} = (ih,jh)$ we must discretize \eqref{eq:contenergy} as
\begin{align}
V(\mathbf{u}) = d(K\mathbf{u},\mathbf{g}) + \alpha J(\mathbf{u}),
\label{eq:discenergy}
\end{align}
where $K: \mathbb{R}^n \rightarrow \mathbb{R}^n$ is a discretization of $K_c$, $\mathbf{u}$ is the output image indexed as $u_{ij}$, and $d$ and $J$ are discretizations of $d_c$ and $J_c$. If $d_c$ is an $L^p$ norm, with $K_c$ a restriction to $\Omega \backslash D$, we discretize it as
\begin{align*}
\left( \int \limits_{\Omega \backslash D} |K_cu - g|^p \mathrm{d} \mathbf{x} \right)^{1/p}
 \approx  \left(h^2\sum_{(i,j) \in \Omega \backslash D}| (K \mathbf{u})_{ij} - g_{ij} |^p \right)^{1/p} =: d(K \mathbf{u}, \mathbf{g}).
\end{align*}
If $J_c$ is on integral form, one can use quadrature to discretize it as
\begin{align*}
J_c(u) = \int \limits_{\Omega} H(u) \mathrm{d}\mathbf{x}
& \approx  h^2\sum_{i,j} H(u) \big|_{\mathbf{x}_{ij}}.
\end{align*}
Since it requires derivatives of $u$, $H(u) = C(u)G(u)$ in  \eqref{eq:EulerElasticaSmooth} must be approximated at the points $\mathbf{x}_{ij}$ by values $H_{ij}(\mathbf{u})$, such that the final discretization becomes
\begin{align*}
J_c(u)   \approx  h^2\sum_{i,j} H_{ij}(\mathbf{u}).
\end{align*} 
For this approximation we use finite differences on a staggered grid as in \cite{ChanKangShen} and \cite{Tai2011}. 
The stencil used for discretizing both $G(u)$ and $C(u)$ is shown in \cref{stencil}. 
The $u_{ij}$ are shown as green squares, and $u_x$ and $u_y$ are approximated by finite differences at red and blue points. With these, we approximate $G(u)$ and $C(u)$.
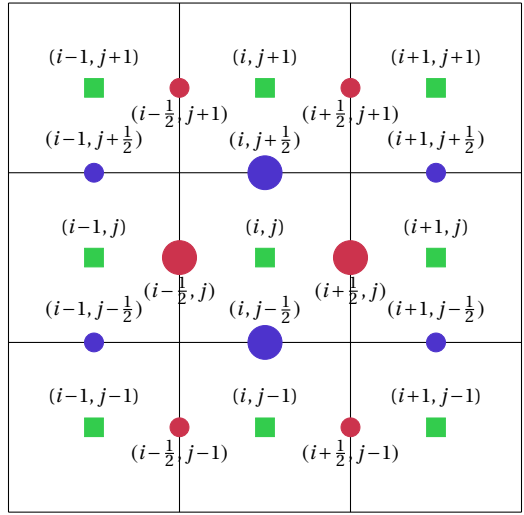
\begin{figure}[ht]
\begin{center}
\begin{tikzpicture}[scale = 0.75]

 \draw (-4.5,-4.5) -- (-4.5,4.5);
  \draw (-1.5,-4.5) -- (-1.5,4.5);
  \draw (1.5,-4.5) -- (1.5,4.5);
  \draw (4.5,-4.5) -- (4.5,4.5);
  \draw (-4.5,-4.5) -- (4.5,-4.5);
  \draw (-4.5,-1.5) -- (4.5,-1.5);
  \draw (-4.5,1.5) -- (4.5,1.5);
  \draw (-4.5,4.5) -- (4.5,4.5);
  
  \stencilpt[bluenode]{-3,1.5}{y1}{{\scriptsize ($i\!-\!1,j \!+\! \frac{1}{2}$)}};
  \bigstencilpt[bluenode]{0,1.5}{y2}{\scriptsize($i , j\! + \!\frac{1}{2}$)};
  \stencilpt[bluenode]{3,1.5}{y3}{\scriptsize($i \!+ \!1, j \!+\! \frac{1}{2}$)};
  \stencilpt[bluenode]{-3,-1.5}{y4}{\scriptsize($i \!- \!1, j \!- \!\frac{1}{2}$)};
  \bigstencilpt[bluenode]{0,-1.5}{y5}{\scriptsize($i,j \!- \!\frac{1}{2}$)};
  \stencilpt[bluenode]{3,-1.5}{y6}{\scriptsize($i\! +\! 1, j\! -\! \frac{1}{2}$)};
  \stencilptund[rednode]{ -1.5,3}{x1}{\scriptsize($i \! - \!\frac{1}{2}, j\! +\!1$)};
  \stencilptund[rednode]{ 1.5,3}{x2}{\scriptsize($i\! + \!\frac{1}{2},j\! + \!1$)};
  \bigstencilptund[rednode]{ -1.5,0}{x3}{\scriptsize($i\! -\! \frac{1}{2},j$)};
  \bigstencilptund[rednode]{ 1.5,0}{x4}{\scriptsize($i\! +\! \frac{1}{2},j$)};
  \stencilptund[rednode]{ -1.5,-3}{x5}{\scriptsize($i\! - \!\frac{1}{2},j\!- \!1$)};
  \stencilptund[rednode]{ 1.5,-3}{x6}{\scriptsize($i\! + \!\frac{1}{2},j\!- \!1$)};
  
  \stencilsq[greennode]{ 0,0}{0}{\scriptsize $(i,j)$};  
  \stencilsq[greennode]{ -3,0}{1}{\scriptsize $(i\!-\!1,j)$};  
  \stencilsq[greennode]{ -3,3}{2}{\scriptsize $(i\!-\!1,j\!+\!1)$};  
  \stencilsq[greennode]{ 0,3}{3}{\scriptsize $(i,j\!+\!1)$};  
  \stencilsq[greennode]{ 3,3}{4}{\scriptsize $(i\!+\!1,j\!+\!1)$}; 
  \stencilsq[greennode]{ 3,0}{5}{\scriptsize $(i\!+\!1,j)$};  
  \stencilsq[greennode]{ 3,-3}{6}{\scriptsize $(i\!+\!1,j\!-\!1)$};  
  \stencilsq[greennode]{ 0,-3}{7}{\scriptsize $(i,j\!-\!1)$};  
  \stencilsq[greennode]{ -3,-3}{8}{\scriptsize $(i\!-\!1,j\!-\!1)$};  
  
\end{tikzpicture}
\caption{Discretization stencil. Green squares: Pixel data $u_{ij}$. All red/blue circles: approximations of $u_x $ and $ u_y$. Large red/blue circles: approximations of $x$ and $y$ components of $\nabla u /|\nabla u|_{\epsilon}$.}
\label{stencil}
\end{center}
\end{figure}
Following the standard approach for $\mathrm{TV}_{\epsilon}$ regularization, $G(u)$ is approximated by backward differences. 
Approximating $C(u)$ requires evaluation of the $x$ and $y$ components of $\frac{\nabla u}{|\nabla u|_{\epsilon}}$ at the large dots (red for the $x$ component, blue for the $y$ component) and taking central differences of these to approximate the divergence. 
To evaluate $|\nabla u|_{\epsilon}$, we approximate values for $u_y$ at the large red dots and $u_x$ at the large blue dots by the mean of the $u_y$ and $u_x$ approximations at the four nearest blue and red points, respectively. In total, the discretized regularizer is
\begin{align}
J(\mathbf{u}) = \sum \limits_{i=1}^{n_x} \sum \limits_{j=1}^{n_y} \left( a + b \left( \delta_{x}^+ \dfrac{\delta_{x}^- u_{ij}}{w_{i-\frac{1}{2},j}} + \delta_{y}^+ \dfrac{\delta_{y}^- u _{ij}}{w_{i,j-\frac{1}{2}}} \right)^2 \right)G_{ij}.
\label{eq:ElasticaRegDisc}
\end{align}
Here, $\delta_x^+$, $\delta_x^-$, $\delta_y^+$, and $\delta_y^-$ denote forward/backward differences in $x$ and $y$ directions,
\begin{align*}
\delta_x^+ f_{ij} = (f_{i+1,j} - f_{ij})/h, \qquad  \delta_x^- f_{ij} = (f_{ij} - f_{i-1,j})/h,\\
\delta_y^+ f_{ij} = (f_{i,j+1} - f_{ij})/h, \qquad  \delta_y^- f_{ij} = (f_{ij} - f_{i,j-1})/h,
\end{align*} 
and the discretization of the $|\nabla u|_\epsilon$ terms depends on the point as
\begin{align*}
G_{ij} = \sqrt{(\delta_x^- u_{ij})^2 + (\delta_y^- u_{ij})^2 + \epsilon},\\
w_{i-\frac{1}{2},j} = \sqrt{(\delta_x^- u_{ij})^2 + (\delta_y^* u_{ij})^2 + \epsilon},\\
w_{i,j-\frac{1}{2}} = \sqrt{(\delta_x^* u_{ij})^2 + (\delta_y^- u_{ij})^2 + \epsilon},
\end{align*}
where
\begin{align*}
\delta_x^* u_{ij} = \dfrac{1}{4}(\delta_x^-u_{i+1,j} + \delta_x^-u_{ij} + \delta_x^-u_{i+1,j-1} + \delta_x^-u_{i,j-1})\\
\delta_y^* u_{ij} = \dfrac{1}{4}(\delta_y^-u_{i,j+1} + \delta_y^-u_{ij} + \delta_y^-u_{i-1,j} + \delta_y^-u_{i-1,j+1}).
\end{align*}
The discrete energy \eqref{eq:ElasticaRegDisc} has a Lipschitz continuous gradient, where the Lipschitz constant depends on $\epsilon$. Thus, any energy of the form  \eqref{eq:discenergy} using \eqref{eq:ElasticaRegDisc} as a regularizer will satisfy \cref{theo:Convergence} when the fidelity term has a Lipschitz continuous gradient.

\begin{figure}[ht]
            \centering
            \includegraphics[width=0.4\linewidth]{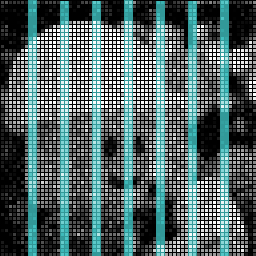}
        \caption{Image split into $M = 8$ and $N = 7$ index sets $\{B_m \}_{m = 1}^8$ and $\{\Gamma_l \}_{l = 1}^7$ (marked cyan).}
        \label{fig:parallelization}
\end{figure}
\subsection{Decoupling and parallelization}
\cref{alg:DG,alg:DG-ADAPT} follow a cyclic ordering with elements updated columnwise, but \cref{theo:Convergence,theo:ConvergenceConv,theo:ConvergenceStrConv} make no assumptions on element ordering, so the convergence rates are unaffected by reordering updates. 
Also, \cref{stencil} indicates that updating $u_{ij}$ will only affect the $H_{\mu \nu}(\mathbf{u})$ with $(\mu,\nu)$ immediately surrounding $(i,j)$. 
Hence, if we split the image in $M + N$ parts as shown in \cref{fig:parallelization} and sweep through the cyan elements first, the blocks separated by cyan pixels can be updated independently of each other and thus in parallel, a domain decomposition strategy similar to that in \cite{xie1999new}.
\begin{algorithm}[DG-PARALLEL]
\begin{algorithmic}
\label{alg:DG-PARALLEL}
\STATE{}
\STATE{$\text{Choose } \tau > 0, \, tol > 0 \text{ and } \mathbf{u}^0 \in \mathbb{R}^n. \text{ Set } k = 0. $ Initialize $M$ threads.}
\STATE{Define $M$ index sets $B_m$ and $N$ index sets $\Gamma_l$.}
\REPEAT 
\STATE{$\mathbf{parallel:} \, N \text{ threads. Thread number } l \text{ does:}$}
\STATE{$\mathbf{v}_0^{k,l} = \mathbf{u}^k$}
\FOR{$j \in \Gamma_l$}
\STATE{Solve $\beta_j^{k}  = -\tau (V(\mathbf{v}_{j-1}^{k,l} + \beta_j^{k} \mathbf{e}_j) - V(\mathbf{v}_{j-1}^{k,l}))/\beta_j^{k}$}
\STATE{$\mathbf{v}_j^{k,l} = \mathbf{v}_{j-1}^{k,l} + \beta_j^{k} \mathbf{e}_j$}
\ENDFOR
\STATE{$\textbf{Reduce: }\, \mathbf{v}_0^{k} = \mathbf{u}^k + \sum \limits_{j \in \cup_{l=1}^{N} \Gamma_l} \beta^k_j \mathbf{e}_j$}
\STATE{$\mathbf{Parallel: } \, M \text{ threads. Thread number } m \text{ does:}$}
\STATE{$\mathbf{v}_0^{k,m} = \mathbf{v}_0^k$}
\FOR{$j \in B_m$}
\STATE{Solve $\beta_j^k  = -\tau(V(\mathbf{v}_{j-1}^{k,m} + \beta_j^k \mathbf{e}_j) - V(\mathbf{v}_{j-1}^{k,m}))/\beta_j^k$}
\STATE{$\mathbf{v}_j^{k,m} = \mathbf{v}_{j-1}^{k,m} + \beta_j^k \mathbf{e}_j$}
\ENDFOR
\STATE{$\textbf{Reduce: }\, \mathbf{u}^{k+1} = \mathbf{v}_0^k + \sum \limits_{j \in \cup_{m=1}^M B_m} \beta^k_j \mathbf{e}_j $}
\STATE{$k = k + 1$}
\UNTIL{$\left(V(\mathbf{u}^k) - V(\mathbf{u}^{k-1})\right)/V(\mathbf{u}^0) < tol$}
\end{algorithmic}
\end{algorithm}
This inspires \cref{alg:DG-PARALLEL}, a parallel version of \cref{alg:DG} where the indices of the unknowns are divided into two collections of index sets, $\{B_m\}_{m=1}^M$ and $\{\Gamma_l\}_{l = 1}^N$, based on the dependency radius of $V$, defined below. Note that the acceleration procedure  proposed in \cref{alg:DG-ADAPT} still works here. For a rigorous discussion, we first introduce a distance measure between index sets. Define the distance between two index pairs $(i,j)$ and $(k,l)$ by
\begin{align*}
\mathrm{dist}_{\text{ind}}\left((i,j),(k,l)\right) = \max\left\lbrace|i-k|,|j-l| \right \rbrace,
\end{align*}
which is the graph distance when every index has edges to the closest indices vertically, horizontally and diagonally. We define the distance between two index sets as
\begin{align*}
\mathrm{dist}_{\text{set}}(I_m,I_n) = \min \limits_{\substack{(i,j) \in I_m\\ (k,l) \in I_n}}\mathrm{dist}_{\text{ind}}\left((i,j),(k,l)\right). 
\end{align*}
If $\mathrm{dist}_{\text{set}}(I_m,I_n) = 0$, then $I_m$ and $I_n$ share at least one index; if $\mathrm{dist}_{\text{set}}(I_m,I_n) = 1$, at least one index in $I_m$ is adjacent to an index in $I_n$, horizontally, vertically or diagonally; if $\mathrm{dist}_{\text{set}}(I_m,I_n) = 2$, there is a band of width 1 of indices separating $I_m$ and $I_n$, et cetera. We can now define the dependency radius of a function; we say that $V: \mathbb{R}^n \rightarrow \mathbb{R}$ has dependency radius $R$ if for all $\delta \in \mathbb{R}$ and $(i,j)$,
\begin{align*}
V(\mathbf{u} + (\delta - u_{ij})\mathbf{e}_{ij}) - V(\mathbf{u}) = F(\mathbf{u}_{ij}^R(\mathbf{u}),\delta),
\end{align*}  
where $F: \mathbb{R}^{(2R+1)^2 + 1} \rightarrow \mathbb{R}$ is a function depending on $\delta$ and
\begin{align*}
\mathbf{u}_{ij}^R(\mathbf{u}) = 
(u_{i_{\min},j_{\min}},...,
u_{ij},...,
u_{i_{\max},j_{\max}}),
\end{align*}
where 
\begin{align*}
i_{\min} &= \max\{i - R,1\},   & j_{\min} &= \max\{j - R,1\}&\\ 
i_{\max} &= \min \{i+R, n_x\}, & j_{\max} &= \min \{j+R, n_y\}.&
\end{align*} 
This means that computing the change in $V$ from updating unknown number $(i,j)$ requires only the unknowns with indices within a distance of $R$. In the discretized Euler's elastica problem we have $R = 1$. If $V$ has dependence radius $R$, then $u_{ij}$ can be updated using the Itoh-Abe discrete gradient independently of $u_{kl}$ if $\mathrm{dist}_{\mathrm{ind}}((i,j),(k,l)) > R$.
We can decouple a problem with dependency radius $R$ by choosing $M$ index sets $B_m \subset \Omega$ such that $\mathrm{dist}_{\text{set}}(B_m,B_n) > R$ for all $(m,n)$. Then, one chooses a second collection of $N$ index sets $\Gamma_l$ such that $\mathrm{dist}_{\text{set}}(\Gamma_k,\Gamma_l) > R$ and $\cup_{l=1}^N \Gamma_l = \Omega / \cup_{m=1}^M B_m$. 
Note that in general, $M \neq N$ and that while this discussion has been focused on two-dimensional indexing, generalizing $\mathrm{dist}_{\text{ind}}$ and $\mathrm{dist}_{\text{set}}$ in the obvious manner to higher-dimensional index pairs admits a similar approach in arbitrary indexing dimensions.

\subsection{Effect of dependency radius on complexity of the algorithm} A consequence of $V$ having dependency radius $R$ is that $\partial V / \partial u_{ij}$ depends on $\mathbf{u}_{ij}^R$ only. 
This can be used to obtain sharper versions of \cref{theo:Convergence,theo:ConvergenceConv,theo:ConvergenceStrConv} through a property presented in the following lemma for two-dimensional indexing. 
\begin{lemma} \label{lemma:MainLemmaSharper}
If \cref{ass:MainAssumption} holds and in addition $V$ has dependency radius $R$, the $\mathbf{u}^k$ produced by \cref{alg:DG-ADAPT} satisfy
\begin{align}
\| \nabla V(\mathbf{u}^{k})\|^2 \leq \nu \left(V(\mathbf{u}^{k}) - V(\mathbf{u}^{k+1})\right), 
\end{align}
with 
\begin{align*}
\nu = 2  L_{\max}^2 \left( (2R+1)^2 \tau_{\max} + \dfrac{\tau_{\max} }{L_{\max}^2 \tau_{\min}^2}  \right).
\end{align*}
\end{lemma}
\begin{proof}
Recall the coordinatewise formulation of the Itoh--Abe scheme, with 2D indexing where, with $\beta_{lm}^k = u_{lm}^{k+1} - u_{lm}^{k}$,
\begin{align*}
 \beta^k_{lm} =  - \tau_k \dfrac{V\left(\mathbf{v}^{k}_{l,m}\right) - V\left(\mathbf{v}^{k}_{l,m-1}\right)}{\beta_{lm}^{k} }. 
\end{align*}
Here,  $\mathbf{v}^{k}_{l,m} := \mathbf{u}^k + \sum_{i=1}^{l-1} \sum_{j=1}^{n_y} \beta_{ij}^{k} \mathbf{e}_{ij} + \sum_{j=1}^m \beta_{lj}^{k} \mathbf{e}_{lj}$.  We follow the proof of \cref{theo:Convergence} up to \eqref{eq:startingpoint}, where we exploit the dependence radius $R$ of $V$. For an $s \in (0,1)$, we have
\begin{align*}
\left| \dfrac{\partial V}{\partial u_{lm}}(\mathbf{u}^{k+1}) \right|&= \left| \dfrac{\partial V}{\partial u_{lm}}(\mathbf{u}^R_{lm}(\mathbf{u}^{k+1})) \right|\\
 &\leq \left| \dfrac{\partial V}{\partial u_{lm}}(\mathbf{u}^R_{lm}(\mathbf{u}^{k+1})) -  \dfrac{\partial V}{\partial u_{lm}}(\mathbf{u}^R_{lm}(\mathbf{v}^{k}_{l,m-1}) + s\beta_{lm}^{k} \mathbf{e}_{lm})  \right|  +  \dfrac{1}{\tau_k}\left| \beta_{lm}^{k} \right|.
\end{align*}
Using coordinatewise Lipschitz continuity, we have
\begin{align*}
\left| \dfrac{\partial V}{\partial u_{lm}}(\mathbf{u}^{k+1}) \right| &\leq L_{lm} \| \mathbf{u}^R_{lm}(\mathbf{u}^{k+1}) - \mathbf{u}^R_{lm}(\mathbf{v}^{k}_{l,m-1}) - s\beta_{lm}^{k} \mathbf{e}_{lm} \| + \dfrac{1}{\tau_k}\left|\beta_{lm}^{k} \right|, 
\end{align*}
and summing up over all coordinates, we get
\begin{align*}
\| \nabla V(\mathbf{u}^{k+1} )\|^2  &\leq \sum_{l= 1}^{n_x} \sum_{m= 1}^{n_y}  \left(L_{lm} \| \mathbf{u}^R_{lm}(\mathbf{u}^{k+1}) - \mathbf{u}^R_{lm}(\mathbf{v}^{k}_{l,m-1}) - s\beta_{lm}^{k}\mathbf{e}_{lm} \| + \dfrac{1}{\tau_k}\left|\beta_{lm}^{k} \right| \right)^2 \\
&\leq 2\sum_{l= 1}^{n_x} \sum_{m= 1}^{n_y} \left(  L_{lm}^2 \left( \sum \limits_{i = l_{\min}}^{l_{\max}}\sum \limits_{j = m_{\min}}^{m_{\max}}{\beta_{ij}^k}^2 \right) + \dfrac{1}{\tau_k^2}\left| \beta_{lm}^{k} \right|^2\right) \\
& \leq 2  L_{\max}^2 \left( (2R+1)^2 + \dfrac{1}{L_{\max}^2 \tau_{\min}^2}  \right) \| \mathbf{u}^{k+1} - \mathbf{u}^{k} \|^2.
\end{align*}
Since $\| \mathbf{u}^{k+1} - \mathbf{u}^{k} \|^2 = \tau_k (V(\mathbf{u}^{k}) - V(\mathbf{u}^{k+1}))$, this concludes the proof.
\end{proof}

\textbf{\textit{Remark:}} This improved estimate affects the complexity of \cref{theo:Convergence,theo:ConvergenceConv,theo:ConvergenceStrConv}, reducing it from a worst-case of $\mathcal{O}(n^{1/2})$ to $\mathcal{O}(R)$ for convex two-dimensionally indexed problems with optimal time steps. Indeed, choosing a constant step size $\tau$ minimizing $\nu$, one obtains
\begin{align*}
\tau = \dfrac{1}{(2R+1) L_{\max} }, \qquad \nu = 4(2R+1)L_{\max}.
\end{align*}
Note that for problems involving discretizations such as the Euler's elastica regularization considered here, $L_{\max}$ may still depend on $n$. 
Regardless, this is an improvement over the $\mathcal{O}(n L_{\max})$ or worse complexity of other coordinate descent methods detailed in \cite{wright2015coordinate}.
For general $D$-dimensional indexing one can expect the complexity to scale as $\mathcal{O}(R^{D/2})$ with optimal step size selection. 
\section{Numerical experiments}
\label{sect:Numerical}

In this section, we first apply Euler's elastica regularization to denoising problems and to image inpainting. The results are compared to those of $\mathrm{TV}_{\epsilon}$ regularization to verify the qualitative improvement of Euler's elastica regularization with \cref{alg:DG} over TV. This is not intended as an account on the competitiveness of Euler's elastica against other regularising procedures in general 
but serves as a proof of concept for the qualitative and algorithmic performance of the discrete gradient approach for Euler elastica when compared to discrete gradient schemes for TV. We further investigate the convergence rate numerically, with varying smoothing constant $\epsilon$ and ordering of the unknowns.
We also verify the predictions of \cref{lemma:MainLemmaSharper}. 
Next, we compare the execution time to another state-of-the-art algorithm for non-convex optimization, the iPiano algorithm \cite{Ochs2014}, and to the gradient descent and Heavy-ball algorithms. Finally, we evaluate the algorithm's sensitivity to the initial guess.

All algorithms were implemented as hybrid MATLAB and C functions using the MATLAB EXecutable (MEX) interface, where critical parts of the code are implemented in C. The tests were executed using MATLAB (2017a release) running on a Mid 2014 MacBook Pro with a four-core 2.5 GHz Intel Core i7 processor and 16 GB of 1600 MHz DDR3 RAM. For the Brent-Dekker algorithm implementation we used the built-in MATLAB function $\mathtt{fzero}$, and block parallelization was done using MATLAB's $\mathtt{blockproc}$ and $\mathtt{parfor}$ functions.

\subsection{Image denoising}
We first consider denoising images. The typical choice of fidelity term is an $L^p$ metric where $p$ depends on the type of noise encountered. The discretized forward operator $K$ is the identity operator. We wish to minimize
\begin{align}
V(\mathbf{u}) = \sum \limits_{i,j}  |u_{ij} - g_{ij}|^p +  J(\mathbf{u}). \label{eq:denoise}
\end{align}

In the first example we have added Gaussian noise with a standard deviation of 0.2, using $p = 2$ for the fidelity term. In the second example we have added impulse noise, randomly setting the values of 25\% of the pixels to either 0 or 1 as depicted in the top pictures in \cref{fig:Louvre}. Impulse noise is removed using $p = 1$ in \eqref{eq:denoise}. Note that the fidelity term with $p = 1$ is non-differentiable and falls outside of the theoretical basis of \cref{sect:theory}. 
However, results in \cite{riis2018geometric} show that the Itoh--Abe scheme will converge to a stationary point when using $p = 1$ in \eqref{eq:denoise}. Therefore, this example is included to see how the method behaves beyond the smooth setting.

\begin{figure}[ht!]
                \centering
        \begin{minipage}{.34\textwidth}
            \centering
            \includegraphics[width=.98\linewidth]{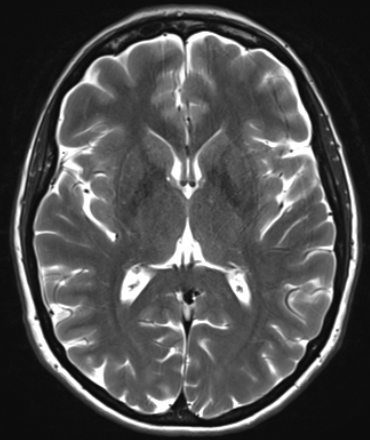}
        \end{minipage}%
        \begin{minipage}{.34\textwidth}
            \centering
            \includegraphics[width=.98\linewidth]{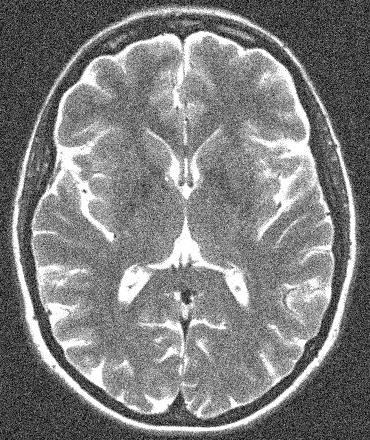}
        \end{minipage}
        \\
        \vspace{3pt}
                \centering
        \begin{minipage}{.34\textwidth}
            \centering
            \includegraphics[width=0.98\linewidth]{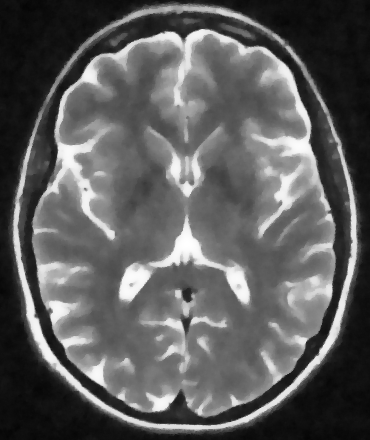}
        \end{minipage}%
        \begin{minipage}{.34\textwidth}
            \centering
            \includegraphics[width=.98\linewidth]{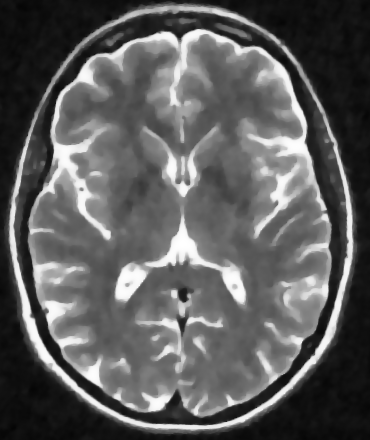}
        \end{minipage}
        \caption{Denoising with $p = 2$. Top left: original image. Top right: noisy input image $g$. Bottom left: $\mathrm{TV}_{\epsilon}$ denoised. PSNR: 20.9421, SSIM: 0.8398. Bottom right: Elastica denoised. PSNR: 21.3760, SSIM: 0.8595.}
        \label{fig:MRI}
\end{figure}

\cref{fig:MRI} shows Euler's elastica denoising with $p = 2$ applied to an image corrupted by Gaussian noise in its lower right hand panel and a $\mathrm{TV}_{\epsilon}$ regularized version in the lower left hand panel. For both $\mathrm{TV}_{\epsilon}$ and elastica denoising, we chose $\epsilon = 10^{-4}$; larger values resulted in blurring and lower values showed no visible improvement but slower convergence. For $\mathrm{TV}_{\epsilon}$ denoising, we set $a = 0.17$, and for elastica denosing, we chose $a = 0.9$ and $b = 0.9$. These values were chosen to maximize PSNR and SSIM.

In the lower right hand panel of \cref{fig:Louvre}, the result of elastica denoising with $p = 1$ on a picture of the Louvre, corrupted by impulse noise, is shown. A $\mathrm{TV}_{\epsilon}$ regularized version is seen in the lower left hand panel. 
As above, we chose $\epsilon = 10^{-4}$ for both $\mathrm{TV}_{\epsilon}$ and elastica denoising. In the $\mathrm{TV}_{\epsilon}$ case, we set $a = 0.8$, and in the elastica case, we chose $a = 0.4$ and $b = 0.2$ to maximize PSNR and SSIM.

\begin{figure}[ht]
                \centering
        \begin{minipage}{.46\textwidth}
            \centering
            \includegraphics[width=.98\linewidth]{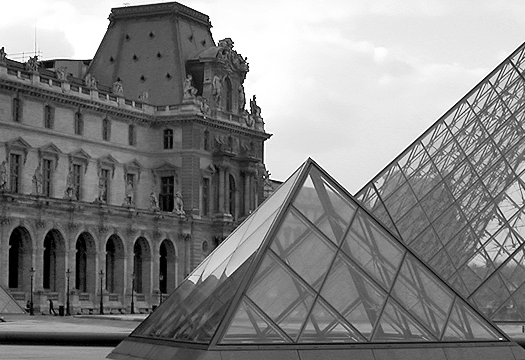}
        \end{minipage}%
        \begin{minipage}{.46\textwidth}
            \centering
            \includegraphics[width=.98\linewidth]{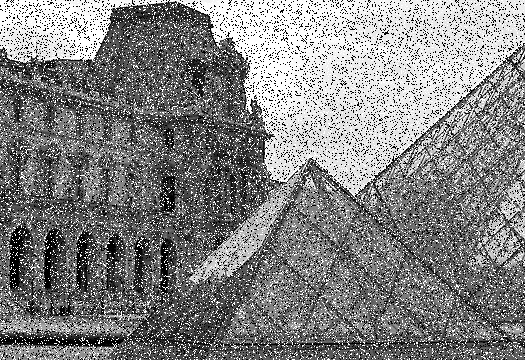}
        \end{minipage}
        \\
        \vspace{3pt}
                \centering
        \begin{minipage}{.46\textwidth}
            \centering
            \includegraphics[width=0.98\linewidth]{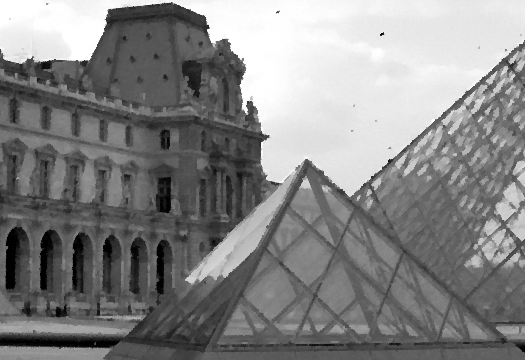}
        \end{minipage}%
        \begin{minipage}{.46\textwidth}
            \centering
            \includegraphics[width=.98\linewidth]{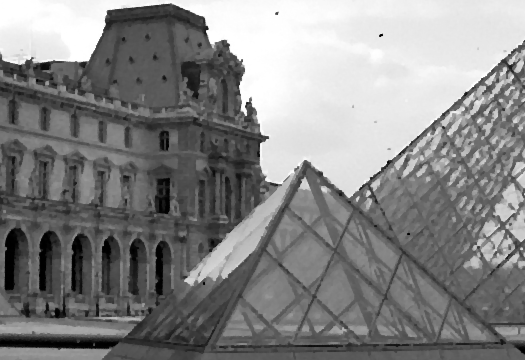}
        \end{minipage}
        \caption{Denoising with $p = 1$. Left: $\mathrm{TV}_{\epsilon}$ denoised. PSNR: 25.7287, SSIM: 0.8572. Right: Elastica denoised. PSNR: 25.9611, SSIM: 0.8628.}
        \label{fig:Louvre}
\end{figure}

As can be seen in both examples, the results of Euler's elastica denoising are slightly more visually appealing than the $\mathrm{TV}_{\epsilon}$ denoised versions, which is as expected since Euler's elastica generalizes TV regularization. 
In \cref{fig:MRI}, edges are sharper and the contrast level better in the elastica denoised image. 
In \cref{fig:Louvre} the pyramids' lines are sharper and the museum's fa\c{c}ade details are clearer. One can also see that the textures are smoother and that edges are less jagged in the elastica reconstruction.

\subsection{Image inpainting}
Inpainting is used when there is data loss in known pixels. Using $p = 2$ and a discretized restriction operator $K$ we wish to minimize
\begin{align*}
V(\mathbf{u}) = \sum \limits_{(i,j) \in \Omega \backslash D}  (u_{ij} - g_{ij})^2 +  J(\mathbf{u}),
\end{align*}
where $D \subset \Omega$ is the damaged domain. \cref{fig:Flower} shows the result of applying \cref{alg:DG} to an example inpainting problem. Here, the top left panel shows the original image, the top right panel the image with 95\% of the pixels removed randomly, the bottom left panel a $\mathrm{TV}_{\epsilon}$ inpainted image, and the bottom right panel an Euler's elastica inpainted image. For both $\mathrm{TV}_{\epsilon}$ and elastica, we chose $\epsilon = 10^{-4}$. In the $\mathrm{TV}_{\epsilon}$ case, we set $a = 2.5 \cdot 10^{-7}$, and in the elastica case, we chose $a = 10^{-6}$ and $ b = 10^{-5}$. These values were chosen to maximize SSIM. Here, the superiority of Euler's elastica is evident, as more details are reconstructed and the image appears sharper than with $\mathrm{TV}_{\epsilon}$ regularization.
\begin{figure}[ht]
                \centering
        \begin{minipage}{.43\textwidth}
            \centering
            \includegraphics[width=0.98\linewidth]{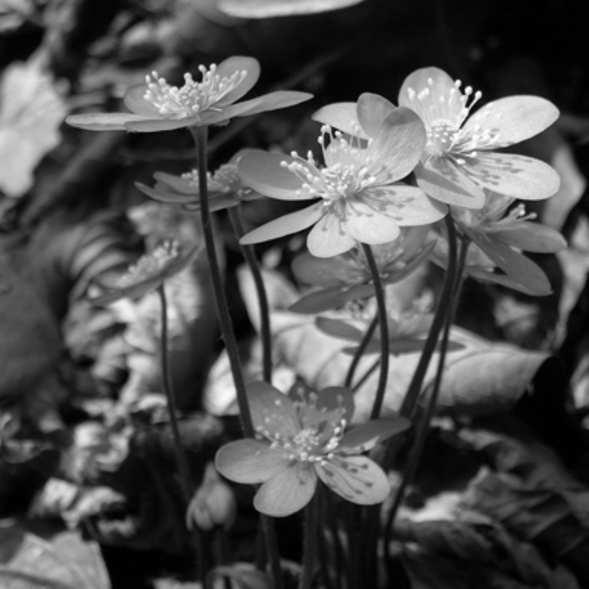}
        \end{minipage}%
        \begin{minipage}{.43\textwidth}
            \centering
            \includegraphics[width=0.98\linewidth]{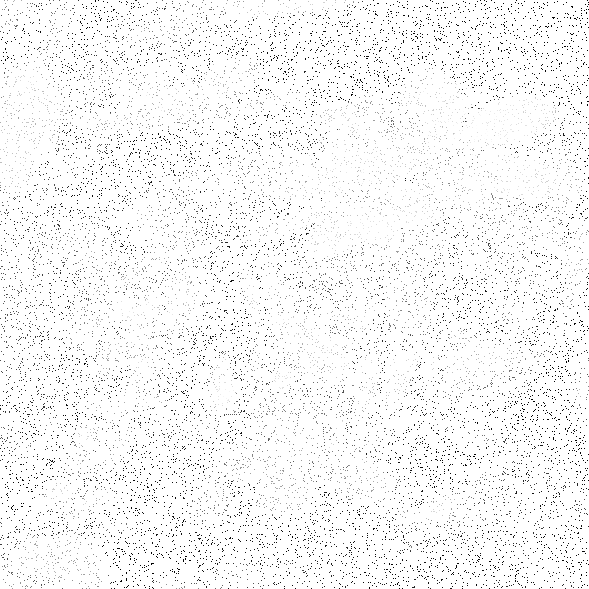}
        \end{minipage}
		\\
		\vspace{3pt}
                \centering
        \begin{minipage}{.43\textwidth}
            \centering
            \includegraphics[width=0.98\linewidth]{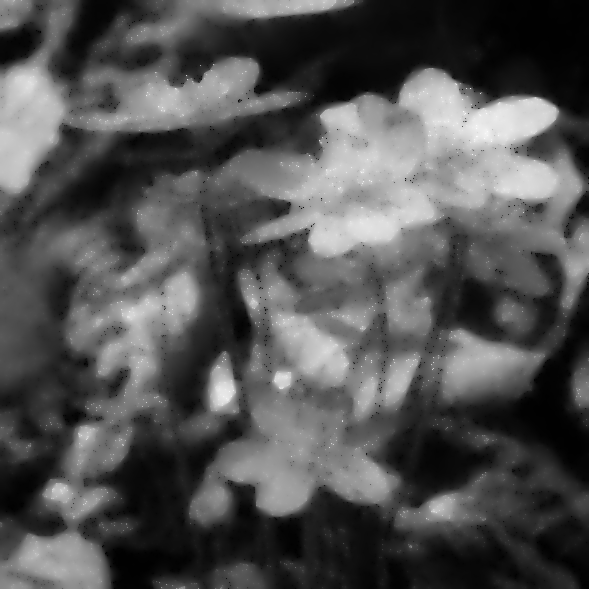}
        \end{minipage}%
        \begin{minipage}{.43\textwidth}
            \centering
            \includegraphics[width=.98\linewidth]{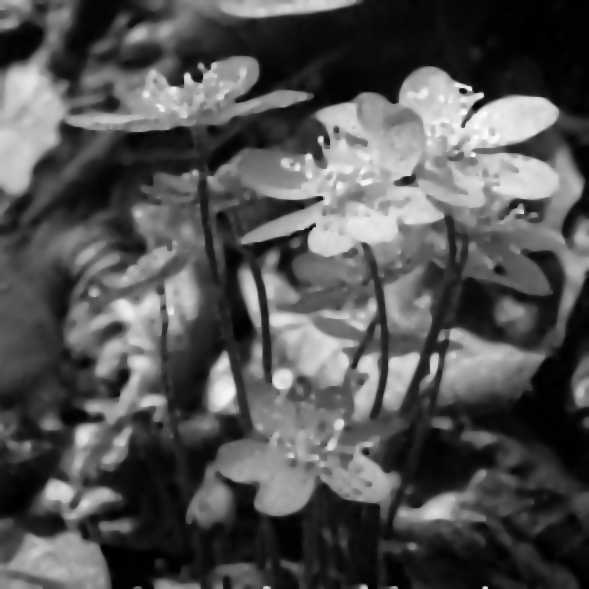}
        \end{minipage}
        \caption{Top left: Original image. Top right: 95 \% random data loss. Bottom left: Inpainted with $\mathrm{TV}_{\epsilon}$ - SSIM: 0.7512. Bottom right: Inpainted with elastica - SSIM: 0.8896.}
        \label{fig:Flower}
\end{figure}

\subsection{Convergence rates}\label{sec:num:convrates}

When denoising using $p = 2$ in \eqref{eq:denoise}, the conditions of \cref{theo:Convergence} are fulfilled, and so we investigate the convergence rates numerically in this case. 
The two plots in \cref{fig:gradconv} show $\min_{0 \leq l \leq k} \| \nabla V(\mathbf{u}^l) \|/\| \nabla V(\mathbf{u}^0) \|$ for DG and DG-ADAPT applied to the Euler's elastica regularized denoising problem with $p = 2$ shown in \cref{fig:MRI} for two choices of $\epsilon$.
Each plot shows the result of initializing with a $\tau_0$ chosen by trial and error to yield the best convergence rate, and with a much smaller $\tau_0$. 
Both algorithms were  started from the same random initalization $\mathbf{u}^0$ and with the same initial time step $\tau_0$. For DG-ADAPT, the additional parameters were chosen as $\rho = 0.99$, $c_1 = 0.7, c_2 = 0.9$ and $\gamma = 1.005$.  
Note that the left hand plot is semilogarithmic while the right hand plot is logarithmic. The left hand plot shows linear convergence for the DG algorithm when $\tau_0$ is chosen correctly and a much slower rate for the suboptimal $\tau_0$. 
The linear convergence can be expected by \cref{theo:ConvergenceStrConv} if the choice of $\epsilon = 10^{-4}$ means $V$, which is twice differentiable, becomes strongly convex in a neighbourhood of the minimizer, or if it is a P\L{}-function. In the right hand plot, where $\epsilon = 10^{-7}$ leads to a more ill-conditioned problem, the convergence rate appears closer to that predicted in \cref{theo:ConvergenceConv}, indicating that a neighborhood of strong convexity has not yet been reached. Both plots show that using adaptive step sizes yields faster convergence than fixed step sizes, especially when $\tau_0$ is not carefully chosen beforehand.

\begin{figure}[ht]
		\begin{minipage}{.45\textwidth}
            \centering
            \includegraphics[width=0.98\linewidth]{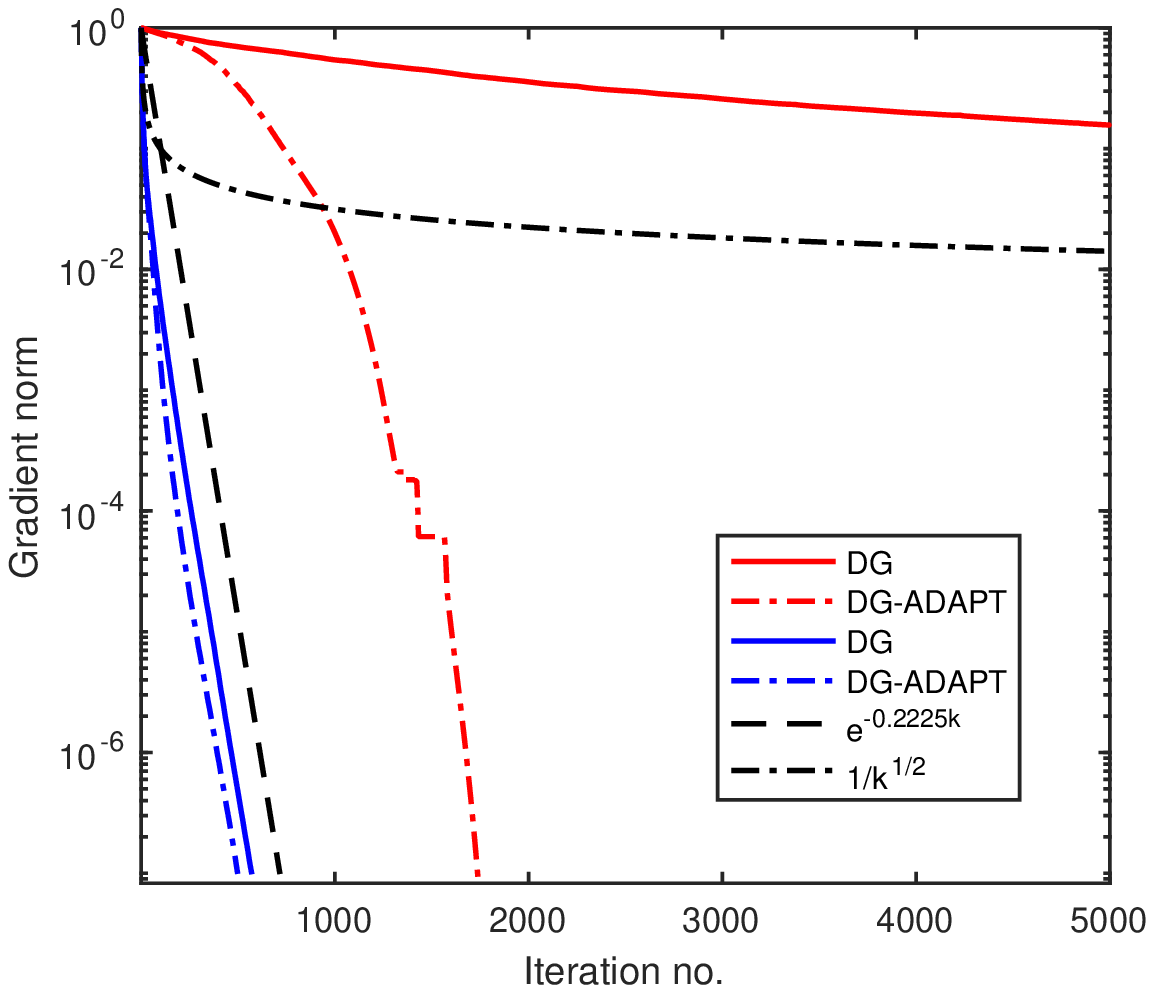}
        \end{minipage}%
\begin{minipage}{.45\textwidth}
            \centering
            \includegraphics[width=.98\linewidth]{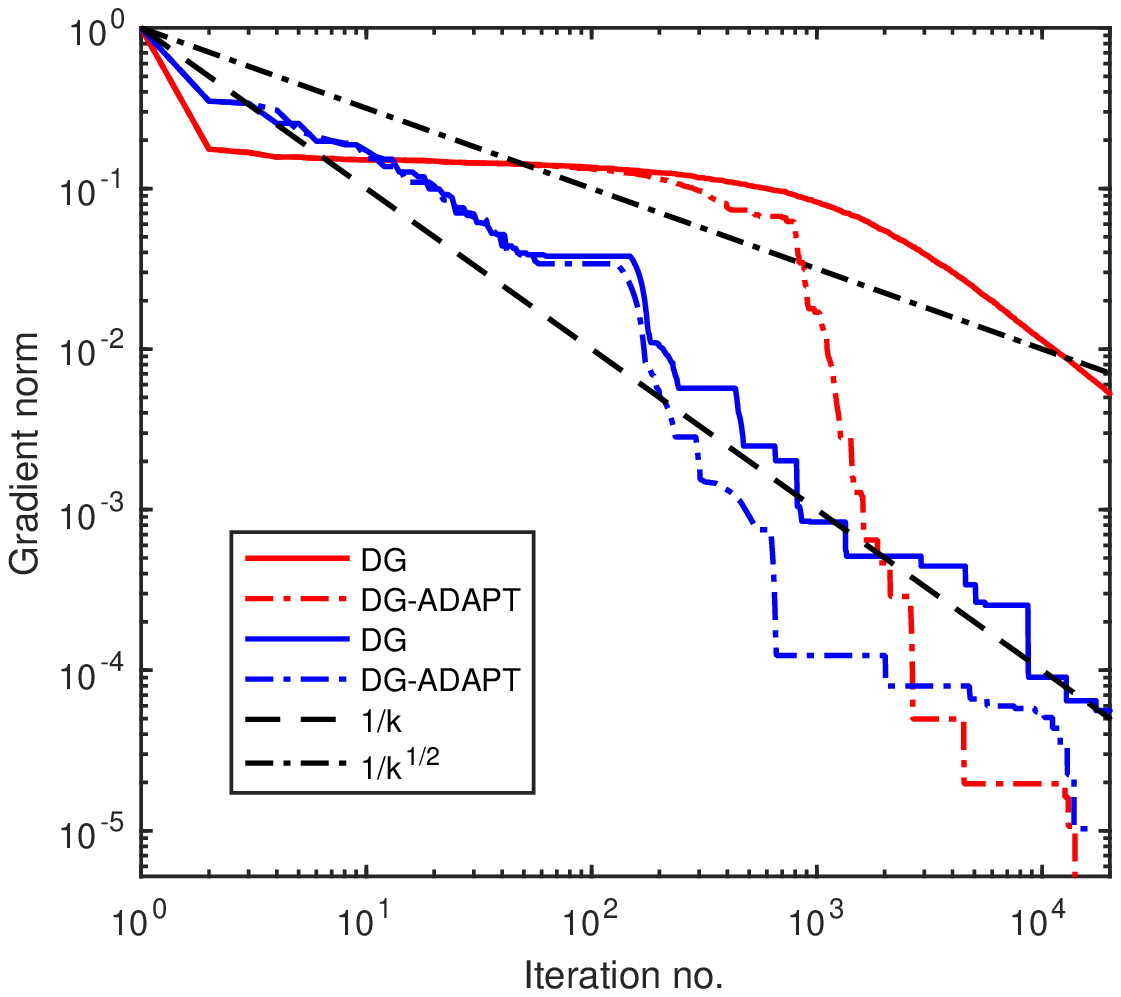}
        \end{minipage}
        \caption{Convergence rates in terms of $\min_{0 \leq l \leq k} \| \nabla V(\mathbf{u}^l) \|/\| \nabla V(\mathbf{u}^0) \|$ for the Euler's elastica regularized denoising problem with $p = 2$ illustrated in \cref{fig:MRI}. Blue denotes $\tau_0 = 0.38$, red denotes $\tau_0 = 0.38 \cdot 10^{-4}$. Left: Using $\epsilon = 10^{-4}$. Right: Using $\epsilon = 10^{-7}$.}
        \label{fig:gradconv}
\end{figure}

\cref{fig:orderconv} shows convergence rates for four different element orderings. The first ordering is the natural ordering which iterates over pixels starting in one corner and proceeding columnwise. 
The second is red-black ordering where pixels $u_{ij}$ with $i + j$ even are updated first, then pixels with $i + j$ odd. 
Third is a random ordering, with the same ordering used for all time steps. 
Last, we consider the block ordering of the parallelized algorithm as illustrated in \cref{fig:parallelization}. 
The plots of \cref{fig:orderconv} concern the same problem as \cref{fig:gradconv}, but with the DG algorithm only and showing rates in terms of the relative optimality error $(V(\mathbf{u}^k) - V^*)/(V(u^0) - V^*)$, where $V^*$ was produced by running the algorithm for 20 000 iterations. 
Step sizes were chosen individually for the different orderings to produce the best possible convergence. 
The left hand plot plateaus at a relative optimality error of $\sim 10^{-16}$, i.e. machine precision. Both plots show that the asymptotic convergence rate, represented by the slope of the lines, is similar for all orderings, as is to be expected from the independency of ordering in the convergence theorems. 
However, the early rate of the random and red-black orderings of the left hand plot is better than that of the natural and block orderings, suggesting that the choice of ordering affects the constant $\nu$ in \cref{theo:Convergence,theo:ConvergenceConv,theo:ConvergenceStrConv}. We consider this an interesting topic for further investigation.

\begin{figure}[ht]
		\begin{minipage}{.45\textwidth}
            \centering
            \includegraphics[width=0.98\linewidth]{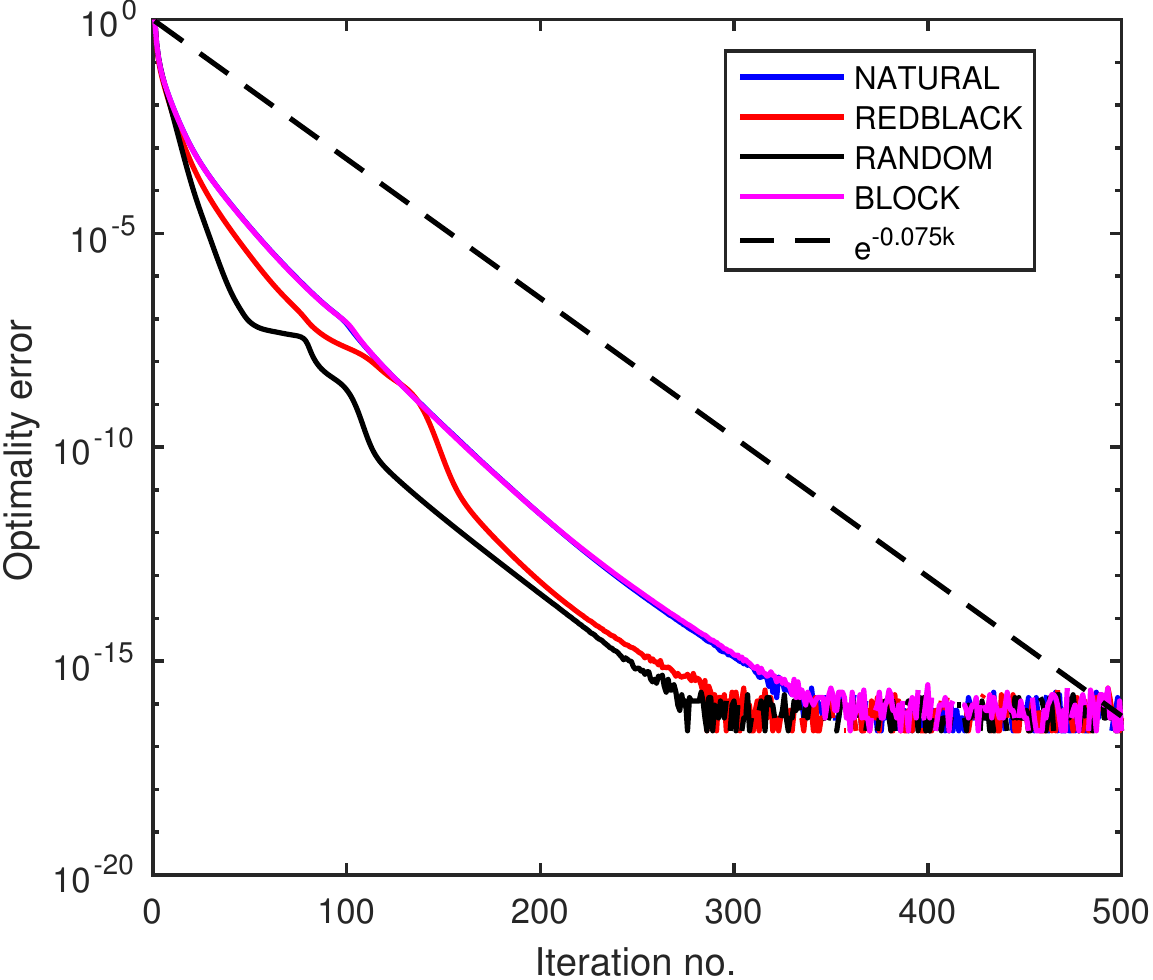}
        \end{minipage}%
\begin{minipage}{.45\textwidth}
            \centering
            \includegraphics[width=.98\linewidth]{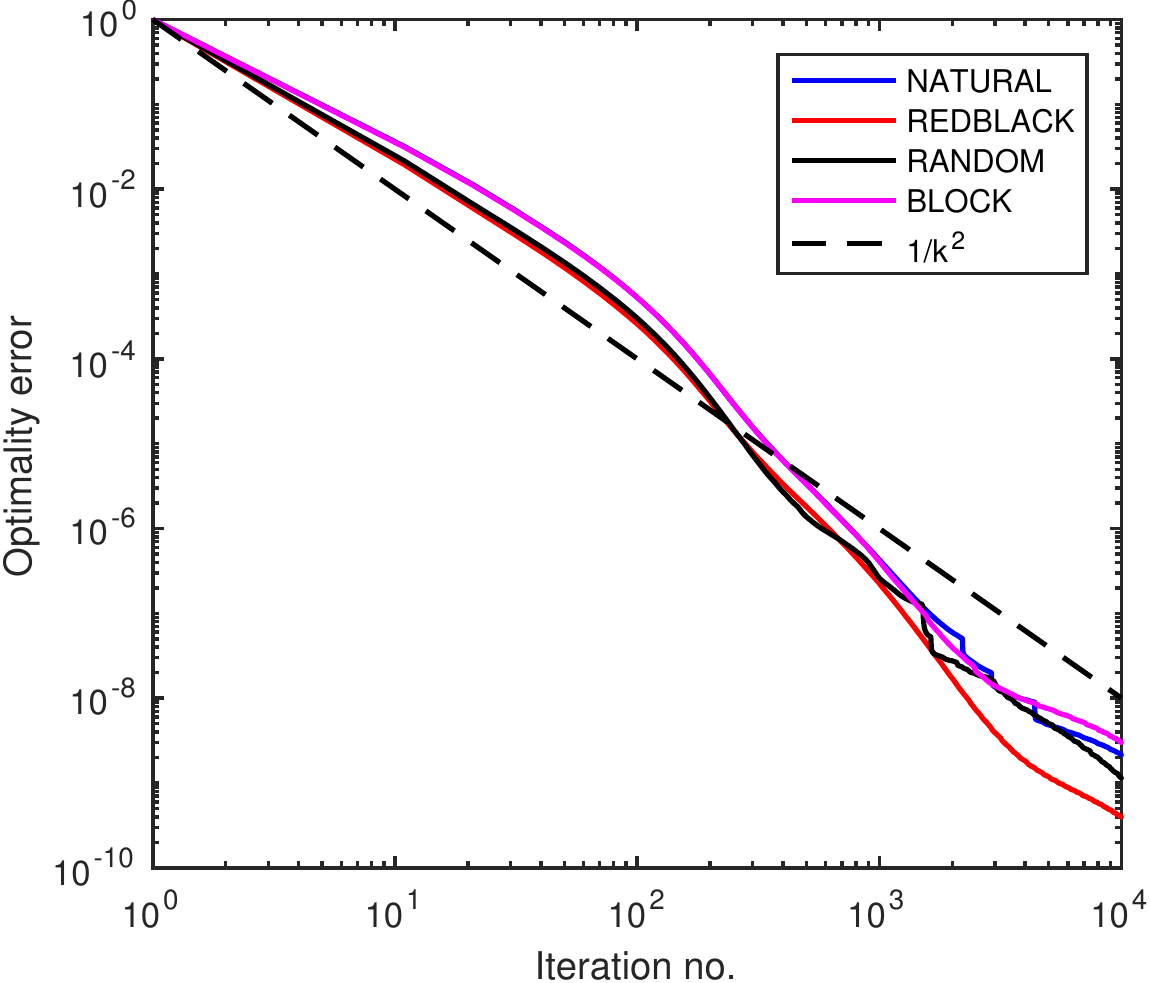}
        \end{minipage}
        \caption{Convergence rates in terms of $(V(\mathbf{u}^k) - V^*)/(V(\mathbf{u}^0) - V^*)$ for the Euler's elastica regularized denoising problem with $p = 2$ illustrated in \cref{fig:MRI}, with different orderings. Left: Using $\epsilon = 10^{-4}$. Right: Using $\epsilon = 10^{-7}$.}
        \label{fig:orderconv}
\end{figure}

\Cref{fig:resolutionconb} shows convergence rates in terms of $(V(\mathbf{u}^k) - V^*)/(V(\mathbf{u}^0) - V^*)$ for a $\mathrm{TV}_\epsilon$ regularized inpainting problem where a square of width 1/2 is missing from the middle of an all-black square image, \textit{i.e.} $\Omega = [0,1]^2$, $D = [1/4,3/4]^2$ and $g = 0$ on $\Omega \backslash D$.
With the parameter choices $a = 1/16$ and $\epsilon = 0.1$, we test with image resolutions  $2^m \times 2^m$, $m = 5,6,7,8,9$ to verify the conclusion of \cref{lemma:MainLemmaSharper}. 
It can be shown that the objective function of the  $\mathrm{TV}_\epsilon$ regularized inpainting problem satisfies the assumptions of \cref{theo:ConvergenceStrConv} with $\sigma = h^2$ and $L_{\max} = h^2(1 + 4a\epsilon^{-1/2}h^{-2})$.
Hence, with $\nu$ chosen optimally as explained in the remark after \cref{lemma:MainLemmaSharper} one should expect the following rate of convergence for the logarithmic error:
\begin{align*}
\log \left( \frac{V(\mathbf{u}^{k}) - V^*}{V(\mathbf{u}^{0}) - V^*}   \right) \leq k\log\left(1 - \dfrac{2\sigma}{\nu} \right) \approx k \frac{2(2R+1)\epsilon^{1/2}}{a}h^2.
\end{align*}
This is observed in \cref{fig:resolutionconb}. For each resolution, the iterations eventually converge at a fixed linear rate. This rate, obtained by exponential fitting, decreases by a factor of 4 as $n = 2^{2m}$ quadruples.

\begin{figure}[ht]
            \centering
            \includegraphics[width=0.5\linewidth]{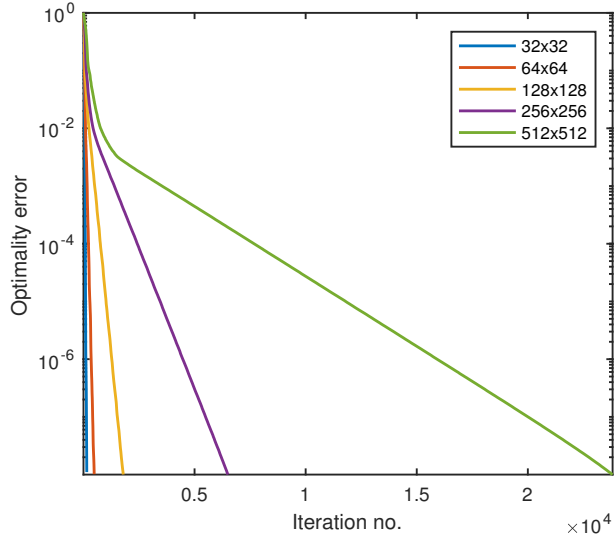}
        \caption{Convergence rates in terms of $(V(\mathbf{u}^k) - V^*)/(V(\mathbf{u}^0) - V^*)$ for the $\mathrm{TV}_{\epsilon}$ square inpainting problem with $p = 2$, for varying problem sizes.}
        \label{fig:resolutionconb}
\end{figure}


\subsection{Execution time} 

As a general algorithm suited to the kind of non-convex minimization problems that the Euler's elastica problem poses, it is reasonable to compare the DG algorithms to the iPiano algorithm of \cite{Ochs2014}; in particular, we use Algorithm 4 from this article. Inspired by Polyak's Heavy-ball algorithm and the proximal gradient algorithm, iPiano considers minimization problems of the form
\begin{align*}
\min_{\mathbf{u} \in \mathbb{R}^n} f(\mathbf{u}) + g(\mathbf{u}),
\end{align*}
where $g$ is convex and possibly non-smooth while $f$ is smooth and possibly non-convex, and iterates based on the update scheme
\begin{align*}
\mathbf{u}^{k+1} = (I + \alpha \partial g)^{-1}(\mathbf{u}^k - \alpha \nabla f(\mathbf{u}^k) + \beta(\mathbf{u}^k - \mathbf{u}^{k-1})),
\end{align*}
where $(I + \alpha \partial g)^{-1}$ denotes a proximal step by
\begin{align*}
(I + \alpha \partial g)^{-1}(\mathbf{y}) = \arg \min_{\mathbf{z} \in \mathbb{R}^n} \dfrac{\| \mathbf{z} - \mathbf{y} \|^2}{2} + \alpha g(\mathbf{z}).
\end{align*}
We wish to time the algorithms on a problem with non-differentiable terms, and so we use a variation on the discrete elastica regularizer \eqref{eq:ElasticaRegDisc}, taking 
\begin{align*}
J(\mathbf{u}) = a\sum \limits_{i=1}^{n_x} \sum \limits_{j=1}^{n_y} \overline{G}_{ij} + b\sum \limits_{i=1}^{n_x} \sum \limits_{j=1}^{n_y} \left( \delta_{x}^+ \dfrac{\delta_{x}^- u_{ij}}{w_{i-\frac{1}{2},j}} + \delta_{y}^+ \dfrac{\delta_{y}^- u _{ij}}{w_{i,j-\frac{1}{2}}} \right)^2 G_{ij} =: aT(\mathbf{u}) + bK(\mathbf{u}) ,
\end{align*}
where
\begin{align*}
\overline{G}_{ij} = \sqrt{(\delta_x^- u_{ij})^2 + (\delta_y^- u_{ij})^2}.
\end{align*}
That is, the TV term $T$ is not differentiable but the curvature term $K$ is. The choice of $f$ and $g$ in the iPiano algorithm depends on whether the fidelity term is differentiable ($p = 2$) or not ($p=1$). We take $f = K + d$ if $p = 2$, but $f = K$ if $p = 1$. Likewise, we take $g = T$ if $p = 2$, but $g = T + d$ if $p = 1$. In both cases, evaluating $(I + \alpha \partial g)^{-1}$ is equivalent to solving a TV regularization problem, which is done efficiently using the Chambolle-Pock algorithm \cite{Chambolle2011}. This algorithm can be accelerated in the case of a uniformly convex fidelity term, i.e. if $d$ is a discrete $L^2$ norm. If it is not, as is the case when $d$ is a discrete $L^1$ norm, no acceleration is possible. The DG-ADAPT algorithm requires the computation of gradients; they are computed using the smoothed \eqref{eq:ElasticaRegDisc}; also, when $p = 1$, the additional smoothing
\begin{align*}
\|u - g \|_{1} = \sum_{i,j}|u_{ij} - g_{ij}| \approx \sum_{i,j}\sqrt{(u_{ij} - g_{ij})^2 + \varepsilon} := \|u - g \|_{1, \varepsilon}
\end{align*}
was used with $\varepsilon = 10^{-12}$ to compute gradients in the DG-ADAPT algorithm.

All algorithms were implemented in serial versions using MEX, and timed. Note that the above non-smoothed TV term was used in the DG/DG-ADAPT algorithms as well for this test. \cref{tab:L2,tab:L1} show timing results for a denoising test on a 512$\times$512 image for different values of $\epsilon$ using a discrete $L^2$ norm and a discrete $L^1$ norm for the fidelity term, respectively. For each $\epsilon$, a reference solution $\bar{V}$ was found by running the DG algorithm for 20 000 iterations or until a minimizer was found with machine precision. The algorithms were then tested on the problem, running until the iterations reached a value of $V(\mathbf{u}^k) \leq 1.0001\cdot\bar{V}$ or 4000 iterations. 
Both algorithms require the solution of a subproblem; a root finding problem for the DG algorithms and the evaluation of a proximal operator for the iPiano algorithm. For the DG algorithms the tolerance of the root finding algorithm was kept at a fixed value while for the iPiano algorithm, the tolerance in the prox operator evaluation was adjusted to obtain the fastest runtime while still converging. 
In DG-ADAPT, the parameter choices $c_1 = 0.7, c_2 = 0.9, \rho = 0.98$ and $\gamma = 1.005$ were used for the $L^2$ test, and $c_1 = 0.2, c_2 = 0.7, \rho = 0.995$ and $\gamma = 1.0025$  were used for the $L^1$ test.

From both tables, it is apparent that the DG and DG-ADAPT algorithms both scale better with $\epsilon$ than iPiano, which reached the maximum number of iterations at $\epsilon = 10^{-6}$ in the $L^1$ test and hence was not timed with $\epsilon = 10^{-7}$. 
However, for larger values of $\epsilon$, iPiano appears to be the better choice.  
The time usage per iteration increases with $\epsilon$ for both iPiano and DG/DG-ADAPT; for iPiano, this is due to the precision in the prox operator evaluation increasing which requires more time. 
For DG/DG-ADAPT, the slight increase can be explained by an increase in the amount of iterations needed by the Brent-Dekker algorithm to solve the scalar subproblems. 
Also note that in \cref{tab:L1}, we can see that DG-ADAPT is not noticeably faster than DG in most cases, indicating that using gradients of smoothed versions of the objective function is not sufficient to accelerate convergence for non-smooth problems.


\begin{table}[tbhp]
{
\caption{Results of $L^2$ test. Format: (Iterations/CPU time (s)). Best times in bold.}
\label{tab:L2}
\begin{center}
\begin{tabular}{l*{3}{c}}
$\epsilon$ & iPiano & DG & DG-ADAPT \\
\hline
$10^{-1}$ & \textbf{30/8.90}    	& 29/15.61 		& 28/17.39 \\
$10^{-2}$ & \textbf{32/9.60}		& 28/14.79 		& 27/16.54 \\
$10^{-3}$ & \textbf{38/12.40	}   	& 38/20.42 		& 35/21.67 \\
$10^{-4}$ & \textbf{59/16.78	}   	& 67/38.22 		& 57/36.62 \\
$10^{-5}$ & \textbf{216/47.74} 	& 115/69.51 		& 89/60.20 \\
$10^{-6}$ & 1684/556.11 	& 180/115.26 	& \textbf{131/91.19} \\
$10^{-7}$ & 3968/1071.37	& 269/196.12  	& \textbf{204/151.71}\\
\end{tabular}
\end{center}
}
\end{table}

\begin{table}[tbhp]
{
\caption{Results of $L^1$ test. Format: (Iterations/CPU time (s)). Best times in bold.}
\label{tab:L1}
\begin{center}
\begin{tabular}{l*{3}{c}}
$\epsilon$ & iPiano & DG & DG-ADAPT \\
\hline
$10^{-1}$ & \textbf{23/21.73}  & 199/171.82 & 168/161.73  \\
$10^{-2}$ & \textbf{59/29.67}  & 288/250.23 & 247/231.56  \\
$10^{-3}$ & \textbf{146/46.36}  & 305/255.12 & 252/231.50  \\
$10^{-4}$ & \textbf{401/229.04} & 300/246.27 & 253/223.21  \\
$10^{-5}$ & 1399/2181.81  & \textbf{303/246.90} & 292/257.47  \\
$10^{-6}$ & 4000/18566.24  & \textbf{303/242.00} & 304/267.11  \\
$10^{-7}$ & N/A  & \textbf{400/335.10} & 423/373.91  \\
\end{tabular}
\end{center}
}
\end{table}


\cref{tab:L2smooth} shows timing results of denoising test problems using the smoothed elastica regularizer \eqref{eq:ElasticaRegDisc}, on a 512$\times$512 image using a squared $L^2$ fidelity term, comparing the DG and DG-ADAPT algorithms to the gradient descent and Heavy-ball algorithms with Armijo step size selection. Here, the parameters of the DG-ADAPT algorithm were chosen as $\rho = 0.99$, $c_1 = 0.7$, $c_2 = 0.9$ and $\gamma = 1.005$ after experimentation. The $\tau$ parameter in the DG algorithm was chosen to give the fastest convergence at $\epsilon = 10^{-4}$, and the same $\tau$ was used as initial $\tau_0$ in DG-ADAPT. The table shows that the DG and DG-ADAPT algorithms outperform the gradient descent and Heavy-ball algorithms on the problem for all $\epsilon$ except $\epsilon = 10^{-1}$, with increasing difference as $\epsilon \rightarrow 0$. The gradient descent algorithm reached the maximum of 4000 iterations at $\epsilon = 10^{-6}$ and was therefore not tested with smaller $\epsilon$.

\begin{table}[tbhp]
{
\caption{Results of $L^2$ test with smooth $\mathrm{TV}_{\epsilon}$ term. Format: (Iterations/CPU time (s)). Best times in bold.}
\label{tab:L2smooth}
\begin{center}
\begin{tabular}{l*{4}{c}}
$\epsilon$ & Gradient Descent & Heavy-ball & DG & DG-ADAPT  \\
\hline
$10^{-1}$ & \textbf{12/5.16}    	 & 23/12.79		 &	20/9.46		& 19/8.61\\
$10^{-2}$ & 27/14.78 	 & 29/17.95 		 &	25/12.72   	& \textbf{24/11.85}\\
$10^{-3}$ & 145/101.00	 & 38/25.16 		 &	32/17.43 	& \textbf{30/15.76}\\
$10^{-4}$ & 388/319.57	 & 96/73.59		 & 	\textbf{42/24.66} 	& 43/24.71\\
$10^{-5}$ & 1310/1271.73	 & 312/284.49 	 &	\textbf{68/43.95	}	& 76/47.92\\
$10^{-6}$ & 4001/4621.38	 & 1018/1079.10 	 & 	119/83.15	& \textbf{112/73.07}\\
$10^{-7}$ & N/A			 & 2922/3765.52	 & 	194/142.05	& \textbf{181/116.88}\\
$10^{-8}$ & N/A			 & 4001/5488.83	 & 	376/291.73	& \textbf{398/260.18}
\end{tabular}	
\end{center}
}
\end{table}

Using the DG-PARALLEL algorithm with column splitting, a speedup of 2-2.5 was observed when employing 4 cores, with larger speedup values for larger images.

\subsection{Dependence on starting point}
To investigate the influence of the starting point on the performance of the algorithm in the minimisation of the non-convex Euler elastica problem, we tested the inpainting problem with DG and iPiano, comparing execution time and reconstruction quality of the two algorithms when starting from a random or unicolor (black) starting image, or from the original image. 

As seen in \cref{fig:inpaintStart}, the DG algorithm produces different but acceptable reconstructions depending on the starting guess. The iPiano algorithm works when starting from a random image and the original image, but is worse with a unicolor starting image. Assuming that starting from the original image gives the best reconstruction, we compare reconstructions starting from other images to this. \Cref{fig:inpaintDiff} shows differences between images obtained starting from the original and the images obtained when starting from random or unicolor images. We can see that the reconstructions are largely in agreement except in certain areas such as the stamens of the top right flower. In \cref{tab:inpaintingTimes} we see that the iPiano algorithm is slower than the DG algorithms when it comes to inpainting, and also that the unicolor initialization produced an answer which was pretty far from optimal as compared to the other initializations.

\begin{table}[ht]
{
\caption{Results of inpainting test. Format: (Final energy/CPU time (s)). Best times in bold.}
\label{tab:inpaintingTimes}
\begin{center}
\begin{tabular}{l*{3}{c}}
Initialization & iPiano & DG & DG-ADAPT \\
\hline
Random      & 0.012334/3434 & \textbf{0.012323/431 } & 0.012323/460 \\
Unicolor    & 0.014892/2934 & 0.012324/2740 & \textbf{0.012324/438} \\
Original    & 0.012263/1131 & \textbf{0.012263/145}  & 0.012263/208 \\
\end{tabular}
\end{center}
}
\end{table}

\begin{figure}[ht]
                \centering
        \begin{minipage}{.35\textwidth}
            \centering
            \includegraphics[width=0.98\linewidth]{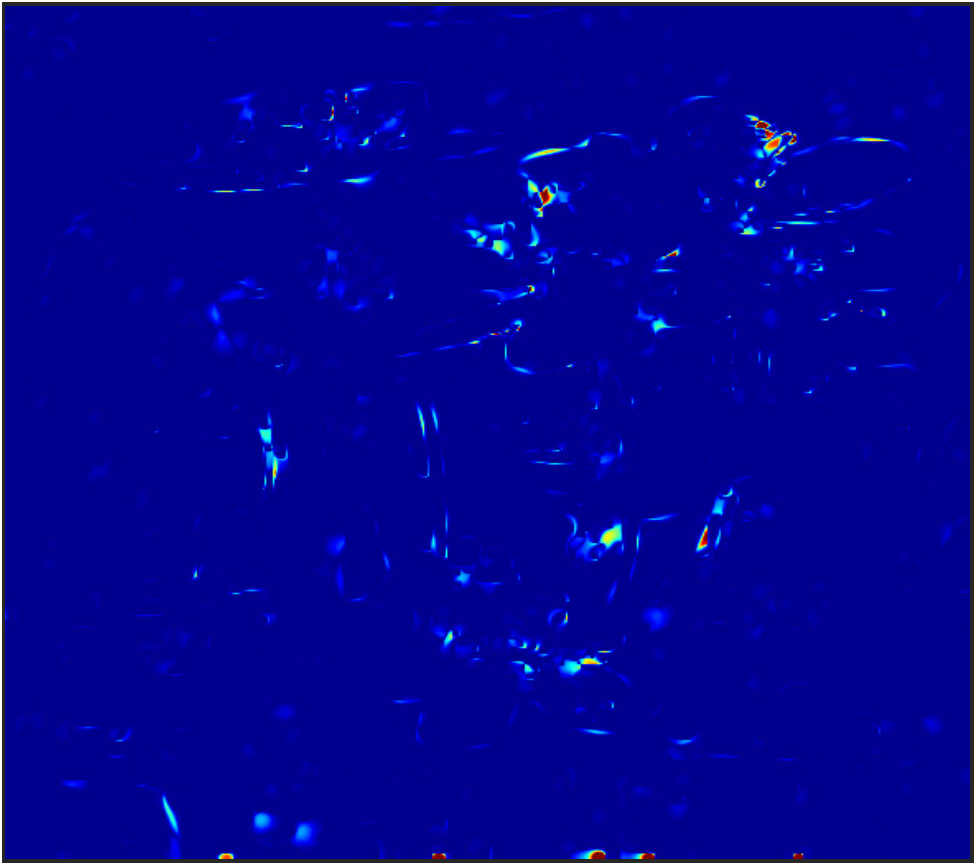}
        \end{minipage}%
        \begin{minipage}{.35\textwidth}
            \centering
            \includegraphics[width=0.98\linewidth]{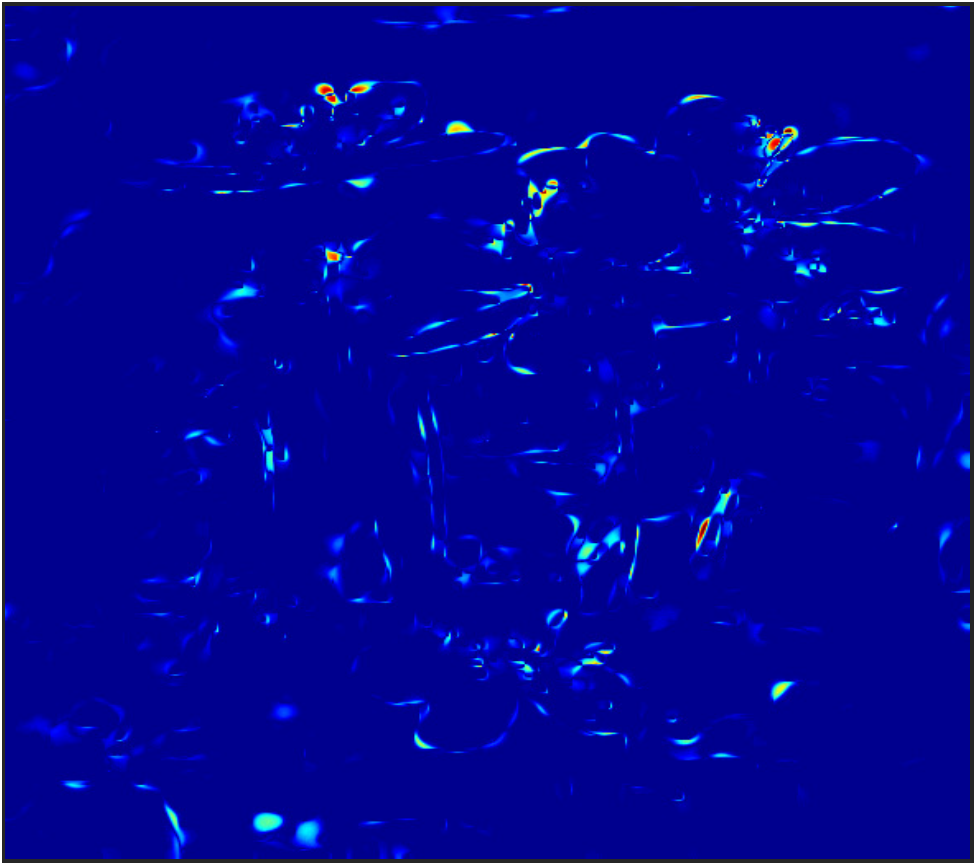}
        \end{minipage}
		\\
		\vspace{3pt}
                \centering
        \begin{minipage}{.35\textwidth}
            \centering
            \includegraphics[width=0.98\linewidth]{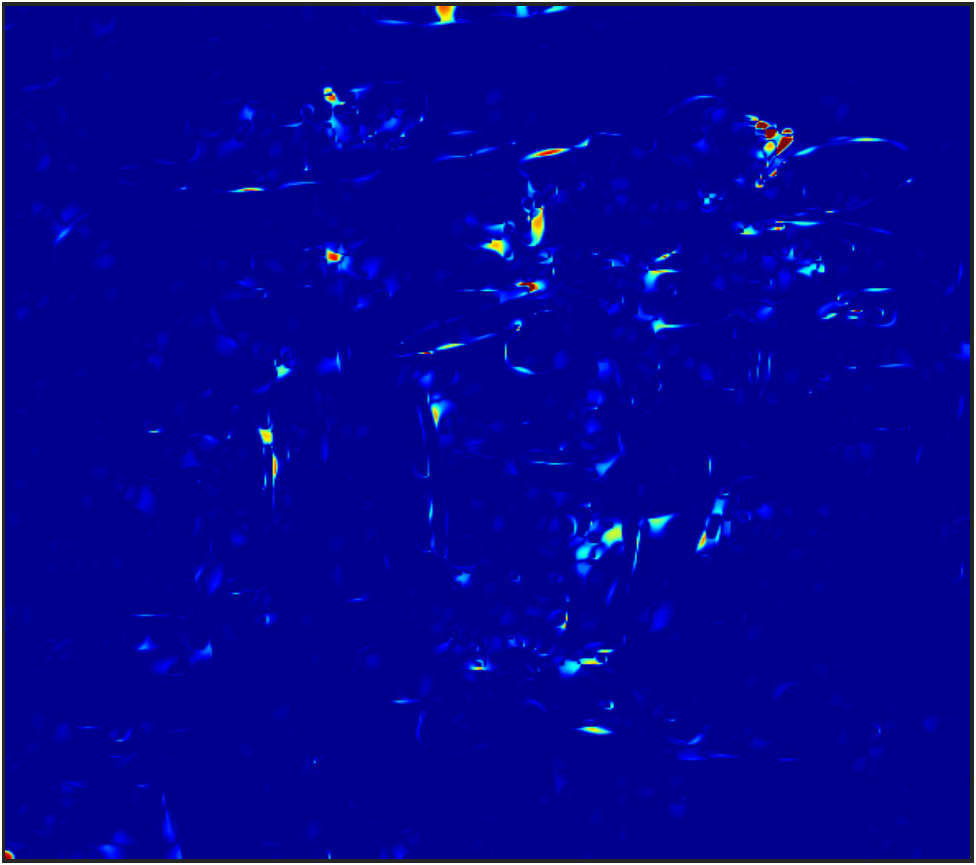}
        \end{minipage}%
        \begin{minipage}{.35\textwidth}
            \centering
            \includegraphics[width=.98\linewidth]{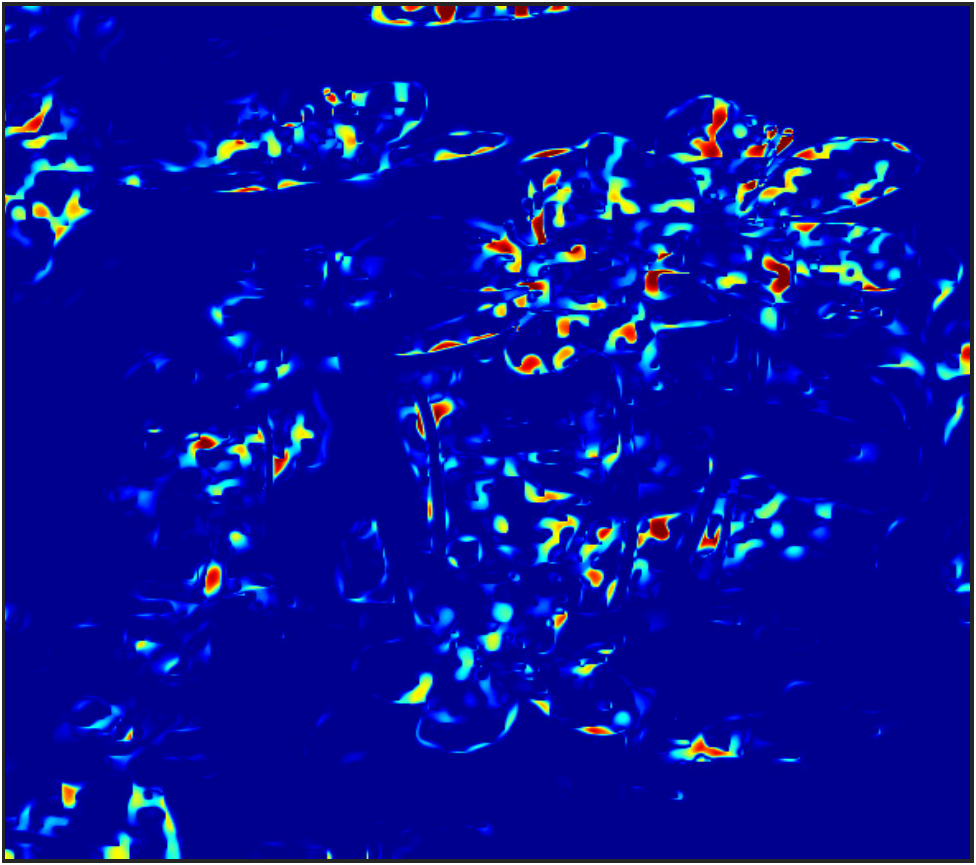}
        \end{minipage}
        \\
        \vspace{2pt}
        \begin{minipage}{\textwidth}
            \centering
            \includegraphics[width=0.8\linewidth]{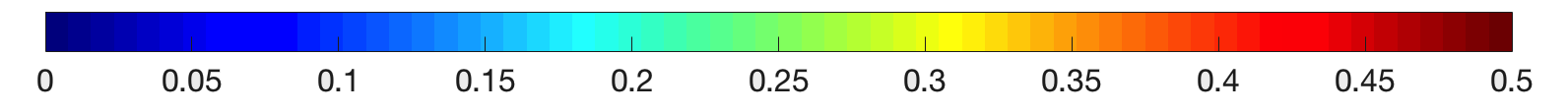}
        \end{minipage}
        \caption{Difference from optimum when inpainting with different starting values. Left column: DG. Right column: iPiano. Top: Random initial value. Bottom: Unicolor nitial value.}
        \label{fig:inpaintDiff}
\end{figure}

\begin{figure}[ht]
                \centering
        \begin{minipage}{.4\textwidth}
            \centering
            \includegraphics[width=0.98\linewidth]{FlowerElastica.png}
        \end{minipage}%
        \begin{minipage}{.4\textwidth}
            \centering
            \includegraphics[width=0.98\linewidth]{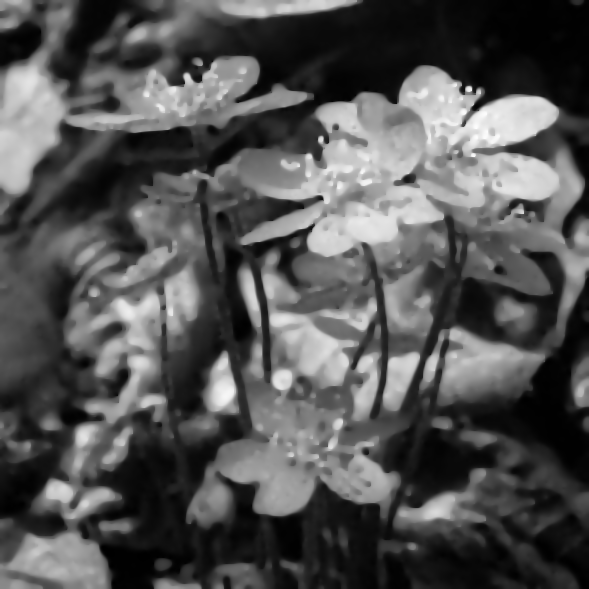}
        \end{minipage}
		\\
		\vspace{3pt}
                \centering
        \begin{minipage}{.4\textwidth}
            \centering
            \includegraphics[width=0.98\linewidth]{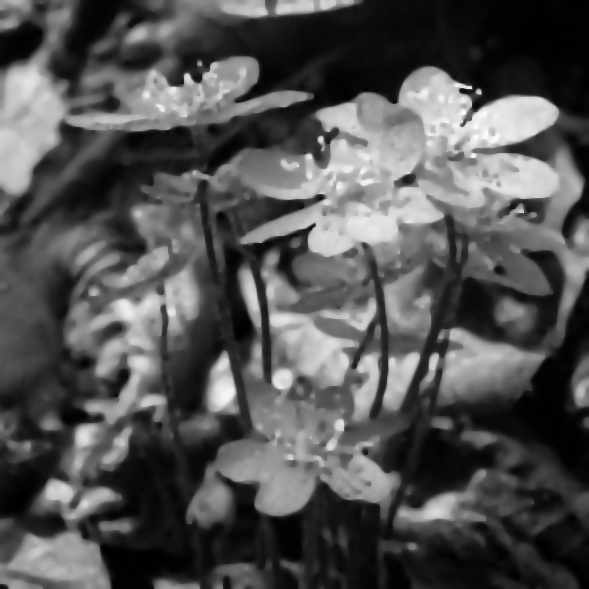}
        \end{minipage}%
        \begin{minipage}{.4\textwidth}
            \centering
            \includegraphics[width=.98\linewidth]{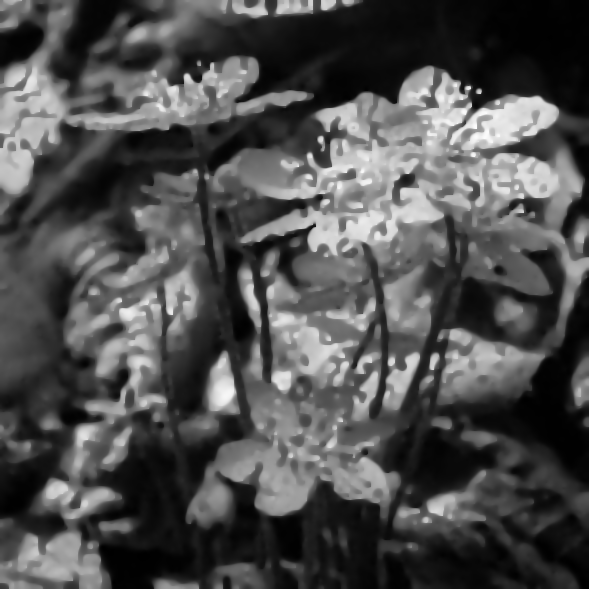}
        \end{minipage}
        \\
		\vspace{3pt}
                \centering
        \begin{minipage}{.4\textwidth}
            \centering
            \includegraphics[width=0.98\linewidth]{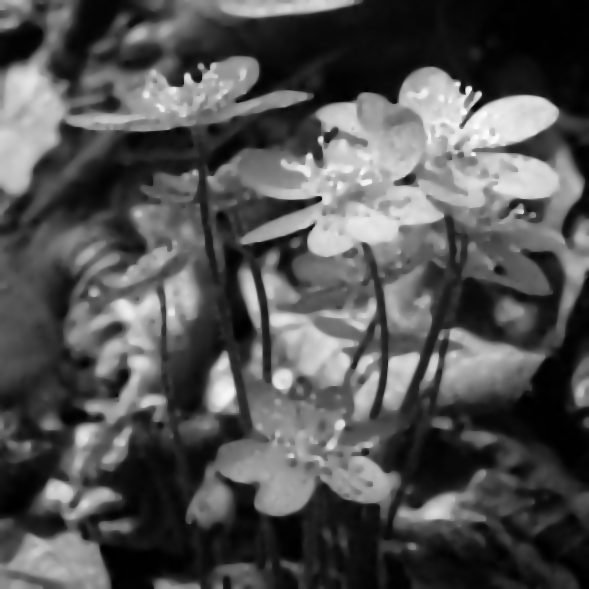}
        \end{minipage}%
        \begin{minipage}{.4\textwidth}
            \centering
            \includegraphics[width=.98\linewidth]{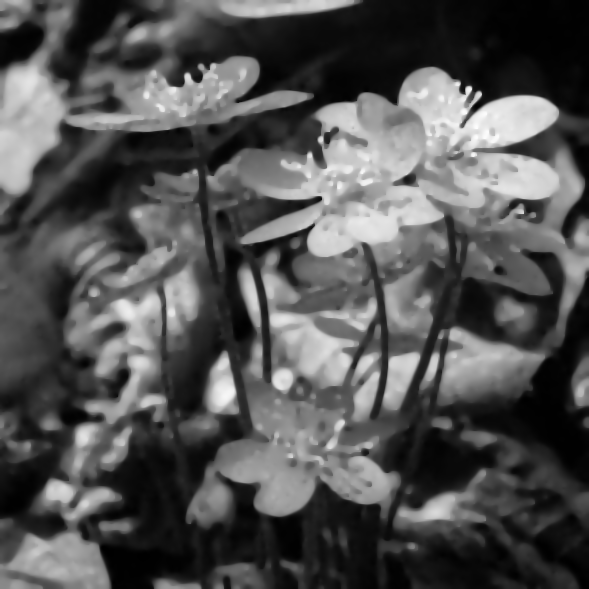}
        \end{minipage}
        \caption{Inpainting with different starting values. Left column: DG. Right column: iPiano. Top: Random initial value. Middle: Unicolor (black) initial value. Bottom: Original initial value.}
        \label{fig:inpaintStart}
\end{figure}

\clearpage

\section{Conclusion}
We have introduced a novel method for solving non-convex optimization problems and tested it on Euler's elastica regularized variational image analysis problems. We have produced a convergence rate estimate for non-convex problems assuming that $V$ is continuously differentiable with Lipschitz continuous gradient. This rate does not depend on the problem size $n$, but rather on the dependency radius $R$ of V. Numerical tests confirm the quality of images denoised and inpainted with Euler's elastica as a regularizer, that the time step adaptivity proposed in \cref{alg:DG-ADAPT} can improve execution time, and that our algorithm performs faster than the iPiano algorithm in certain instances.

There are still open questions, two of which carry special importance. Firstly, it should be possible to improve upon the time step adaptivity of \cref{alg:DG-ADAPT} which, while effective in some instances, is rudimentary. It may be possible to employ a stochastically ordered version as in \cite{CoordDescNesterov} instead. Secondly, one should investigate the convergence properties of \cref{alg:DG} when applied to non-differentiable problems since it is still applicable then. One may also generalize \cref{alg:DG} to a manifold setting using the tools developed in \cite{Celledoni2014977}. Finally, one may wish to apply the discrete gradient approach to other non-convex optimization problems.
\section*{Acknowledgements}
CBS acknowledges the support of the Engineering and Physical Sciences Research Council (EPSRC) 'EP/K009745/1', the Leverhulme Trust project 'Breaking the non-convexity barrier', the EPSRC grant 'EP/M00483X/1', the EPSRC centre 'EP/N014588/1', the Alan Turing Institute 'TU/B/000071', the Isaac Newton Institute, the Cantab Capital Institute for the Mathematics of Information and from NoMADS (Horizon 2020 RISE project grant). This work was supported by the European Union's Horizon 2020 research and innovation programme under the Marie Sklodowska-Curie grant agreement No. 691070.
The authors wish to thank Antonin Chambolle, Elena Celledoni, Brynjulf Owren, and Reinout Quispel for their helpful discussions during this work. We also wish to thank the anonymous referees for their thorough advice.
\bibliography{EP,tb}
\bibliographystyle{siamplain}

\end{document}